\documentclass[review]{elsarticle}
\usepackage{graphicx}
\usepackage[colorlinks]{hyperref}
\hypersetup{pdfauthor={Name}}
\usepackage{lineno}
\usepackage{float}
\usepackage{multicol,multirow}
\usepackage{multienum}
\usepackage{eucal}
\usepackage[T1]{fontenc}
\usepackage[utf8]{inputenc}
\usepackage[english]{babel}
\usepackage{mathtools}
\usepackage{amsmath,amssymb}

\usepackage{graphics}

\usepackage{caption}
\usepackage{subcaption}

\usepackage{xcolor}
\usepackage[hmargin=1in,vmargin=1in]{geometry}

\allowdisplaybreaks

\makeatletter
\newcommand{\subjclass}[2][2020]{%
	\let\@oldtitle\@title%
	\gdef\@title{\@oldtitle\footnotetext{#1 \emph{Mathematics subject classification}: #2}}%
}
\newcommand{\keywords}[1]{%
	\let\@@oldtitle\@title%
	\gdef\@title{\@@oldtitle\footnotetext{\emph{Keywords}: #1}}%
}
\makeatother

\newtheorem{theorem}{Theorem}

\newdefinition{remark}{Remark}
\newproof{proof}{Proof}

\journal{Appropriate journal to be identified}

\begin{document}
	
	\begin{frontmatter}
		\title{The time fractional order derivative for multi-class AR model}
		\tnotetext[t1]{This article is a result of the PhD program funded by Makerere-Sida Bilateral Program under Project 316 "Capacity Building in Mathematics and Its Applications".}
		
		\author[mymainaddress,mysecondaryaddress]{Nanyondo Josephine\corref{mycorrespondingauthor}}
		\cortext[mycorrespondingauthor]{Corresponding author}\ead{nanyondojose@gmail.com}
		
		\author[mysecondaryaddress]{Joseph Y. T. Mugisha}
		\ead{joseph.mugisha@mak.ac.ug}
		
		\author[mysecondaryaddress]{Henry Kasumba}
		\ead{henry.kasumba@mak.ac.ug}
		
		\address[mymainaddress]{Department of Mathematics, Faculty of Science and Education, Busitema University, P.O Box 236, Tororo, Uganda}			
		\address[mysecondaryaddress]{Department of Mathematics, School of Physical Sciences, Makerere University, P.O Box 7062, Kampala, Uganda}
		
		\begin{abstract}
			In this paper, a  multi-class Aw-Rascle \textrm{(AR)} model with time fractional order derivative is presented. The conservative form of the proposed model is considered for the natural extension and generalization of equations involved. The fractional order derivative involved in the model equations is computed by applying the Caputo fractional derivative definition. An explicit difference scheme is obtained through finite difference method of discretization. The scheme is shown to be consistent, conditionally stable and convergent. Numerical flux is computed by original Roe decomposition and an entropy condition applied to the Roe decomposition. From numerical results, the effect of fractional-order derivative of time, on the traffic flow of vehicle classes is determined. Results obtained from the proposed model remain within limits therefore, they are realistic.
			
		\end{abstract}
		\begin{keyword}
			Fractional Aw-Rascle model; Time fractional
			derivative; Explicit difference scheme; Stability; Convergence; heterogeneous vehicular traffic
		\end{keyword}
	\end{frontmatter}
	\section{Introduction}\label{sec:Intro}
	Fractional calculus (FC) refers to the field of mathematics that studies the generalization of classical concepts of differential and integral operators \cite{zhang2019difference,khan2020fractional,liao2019conservative, zhao2016research}. 
	The generalization can lead to more adequate modelling of real processes in order to satisfy the more strict requirements like artificial and natural factors for the road traffic \cite{zhao2016research, bhrawy2015method,flores2016using}. 
	
	The most important characteristic of describing models using fractional-order system is the non-local property, which suggests that the future state depends on the present state and all the previous states \cite{khan2020fractional} 
	It is an infinite dimensional filter (considers the full memory effect \cite{kumar2018efficient}
	due to the fractional order in the differentiator and integrator unlike the integer order, which has limited memory (it is finite dimensional). Motivated by these properties, research and application of fractional differential equations for instance in biology \cite{khan2020fractional}, 
	chemical reactions, water pollution, vehicular traffic flow attracted attention of scholars \cite{zhang2019difference}. 
	However, application of fractional calculus to second-order macroscopic traffic flow models has been neglected.
	
	In order to obtain solutions of the fractional order equations, analytical and approximation methods have been developed \cite{zhang2019difference}. 
	For example, \cite{jassim2017approximate} and \cite{kumar2018efficient} used local fractional Adomian decomposition method, local fractional series expansion method, local fractional homotopy perturbation Sumudu transform scheme and local fractional reduced differential transform method to solve  the fractal Lighthill-Whitham-Richards (LWR) model. \cite{wang2014fractal} applied local fractional conservation laws to investigate the fractal dynamical model of vehicular traffic. These too worked on the LWR model.  \cite{flores2016using} used a fractional order controller to deal with non-modeled dynamics of heterogeneous vehicle strings.  
	
	It is paramount to remark the fact that heterogeneous traffic is a complex system which requires fractional calculus since it involves nonlinear fluid mechanics characteristics \cite{zhao2016research}. For instance in heterogeneous traffic flow, the reaction time of each driver is usually non uniform \cite{cao2016fractional}. However, the existing second- and higher order traffic flow models treat reaction time of drivers to be uniform. In this study, a time fractional order derivative that caters for variations in reaction times of drivers is introduced into the time scale of the proposed model. 
	The rest of the paper is organized as follows. The proposed time-fractional multi-class AR model is given in Section $2.$ The explicit difference scheme is presented in Section $3.$ In Section $4,$ numerical simulations are discussed. The conclusions are presented in Section $5.$
	
	\section{Model formulation}
	In this section, we extend the multi-class AR model was proposed by \cite{nanyondo2024analysis}, by introducing fractional orders into the time derivative term in order to obtain the following proposed model: 
	\begin{equation}
		\dfrac{\partial^{\alpha} U}{\partial t^{\alpha}} + \dfrac{\partial f(U)}{\partial x} = S(U),
		\label{TimeFractionalModel} \end{equation}
	where $\alpha \in (0, 1]$ refers to fractional order of the time derivative.
	The conservative form has been considered for its simplicity to apply fractional calculus. Otherwise, it is difficult to apply product rule on the multi-class AR traffic model. The proposed model variables and parameters are given and defined as follows: 
	\begin{equation*} 
		U = \begin{bmatrix}
			\rho_m \\
			X_m \\
			\rho_c \\
			X_c\end{bmatrix}, ~ 
		f(U) = 
		\begin{bmatrix}
			X_m - \rho_m p_m \\
			\dfrac{X_m^2}{\rho_m} - p_mX_m\\
			X_c - \rho_c p_c \\
			\dfrac{X_c^2}{\rho_c} - p_cX_c
		\end{bmatrix} ~\text{and}~ 
		S(U) = 
		\begin{bmatrix}
			0 \\
			\dfrac{\rho_m}{\tau_m}\left(v_{em} - v_m\right)\\
			0 \\
			\dfrac{\rho_c}{\tau_c}\left(v_{ec}  - v_c\right)
		\end{bmatrix}, \label{VectorForm2}
	\end{equation*}
	where $X_m = \rho_m(v_m + p_m),~ X_c = \rho_m(v_c + p_c)$ and all parameters are positive. Variables $ \rho_c = \rho_c(x,t)$ refer to density of cars, $\rho_m = \rho_m(x,t)$ density of motorcycles, $v_m = v_m(x,t)$ average velocity of motorcycles and $v_c = v_c(x,t)$ the average velocity of cars. The respective ``pressure'' functions are given by
	\begin{equation*}
		p_m = \left(\psi_m \rho_m\right)^{\gamma_m} ~\text{and} ~p_c = \left(\psi_c \rho_c\right)^{\gamma_c},
	\end{equation*} 
	where parameters $\psi_m,~ \psi_c$ are given by  
	\begin{equation*}
		\psi_m = \dfrac{w_m(l_c(1- \delta) + \dfrac{\delta l_m}{3})}{W \delta},~
		\psi_c = \dfrac{w_c}{W(1- \delta)}\left((1- \delta)l_c + \dfrac{1}{3} \delta l_m\right).
	\end{equation*}
	Other parameters $\gamma_m,~ \gamma_c$ represent ``pressure'' exponents of motorcycles and cars while $\tau_m,~ \tau_c$ are the relaxation times. Note that ``pressure'' here does not mean the real pressure but a pseudo-pressure \cite{rascle2002improved}. We adopt a modified version of the Greenshield's equilibrium velocity 
	\begin{eqnarray*}
		\begin{cases}
			v^e_m = v_{m \max}\left(1 - \dfrac{AO_m}{AO_m^{\max}}\right),~ \text{if}~ AO_m \leq AO_m^{\max},\\[1.5ex] 
			v^e_c = v_{c \max}\left(1 - \dfrac{AO_c}{AO_c^{\max}}\right),~ \text{if}~ AO_c \leq AO_c^{\max}, \\[1.5ex]
			v^e_i = 0 ~ \text{otherwise,} ~i=m,c,
		\end{cases}
	\end{eqnarray*}  where $AO_{m},~  AO_c,~ AO_m^{\max},~ AO_c^{\max},~ v_{m\max},~ v_{c\max}$ refer to the respective area occupancy expressed in terms of proportional densities, maximum area occupancy and maximum velocity. Velocities are defined in terms of $AO$ since it is a standard measure of heterogeneous traffic concentration that lacks lane discipline. 
	
	Suppose the proposed model (\ref{TimeFractionalModel}) is defined in the domain $[0,I] \times [0,K]$ on the $x-t$ plane, with initial-boundary conditions:
	\begin{eqnarray}
		U(x, 0) = g_1(x),~ 0 \leq x \leq I, \label{InitialCondition}\\
		U(0, t) = g_2(t),~ U(I,t) = g_3(t),~  0 \leq t \leq K, \label{BoundaryConditions}
	\end{eqnarray} in which 
	\begin{equation}
		0 < \alpha \leq 1
	\end{equation}
	and 
	\begin{equation}
		g_1(x),~ g_2(t),~ g_3(t)
	\end{equation} are continuous functions. The time fractional derivative term $\dfrac{\partial^\alpha U}{\partial t^\alpha}$ is defined as the Caputo fractional derivative \cite{zhang2019difference} because it has physically interpretable initial conditions \cite{zhang2014time}. It is given by
	\begin{equation}\dfrac{\partial^\alpha U}{\partial t^\alpha} = \dfrac{1}{\Gamma (1 - \alpha)} \displaystyle\int_{0}^{t} (t - \eta)^{-\alpha} \dfrac{\partial U(x, \eta)}{\partial \eta} d \eta. \end{equation} 
	In the next section, we follow works by \cite{guo2015fractional,ali2016solutions,zhang2014time,zhang2019difference} to seek a
	numerical solution of the proposed model (\ref{TimeFractionalModel}) - (\ref{BoundaryConditions}) by adopting an explicit difference scheme of discretization.  
	
	\section{The explicit difference scheme}\label{sec:ExplicitDiffernceScheme}
	The proposed model (\ref{TimeFractionalModel}) in a space-time domain $\displaystyle{[0,I] \times [0,K]}$ is considered. Solving it would mean evolving the solution $U(x,t)$ in time starting from initial condition $U(x, 0)$ and subject to boundary conditions. The space and time grids are divided into $I$ and $K$ equally spaced grid points $x_i = i \Delta x,~ i = 0, \ldots, I$ and $t^k = k \Delta t,~ k = 0 , \ldots, K,$ of length $\Delta x$ and $\Delta t,$ respectively. Approximate or discrete values of $U(x,t),$ for the data given by (\ref{TimeFractionalModel}) at a time step $k$ and spatial position $i,$ are represented by $U_i^k \equiv U \left(i \Delta x, k \Delta t \right) \equiv U \left(x_i, t^k \right).$ Next, periodic conditions are applied to the left and right boundary points $U_0^k$ and $U_I^k,$ respectively. Since at any time step $k,$ the numerical solutions are obtained as either initial data values or values computed in a previous time step, an $L_1-$algorithm (see the following Equation \ref{TimeFracCaputo}) and Roe's scheme are applied to approximate Caputo time fractional derivative and the numerical flux of (\ref{TimeFractionalModel}). Hence, the solution $U_i^{k+1},$ at the next time step is obtained as follows: 
	\begin{eqnarray}
		\dfrac{\partial^\alpha U}{\partial t^\alpha} &=& \dfrac{1}{\Gamma (1 - \alpha)} \displaystyle \sum_{j=0}^{k}\int_{t_j}^{t_{j+1}} (t_{k+1} - \eta)^{-\alpha} \dfrac{\partial U(x, \eta)}{\partial \eta} d \eta + O(\Delta t), \nonumber \\
		&\approx& \dfrac{\Delta t^{1-\alpha}}{(1 - \alpha)\Gamma (1 - \alpha)} \displaystyle \sum_{j=0}^{k}
		\left[(k+1-j)^{1-\alpha} - (k-j)^{1-\alpha}\right] \left[\dfrac{U(x_i, t_{j+1}) - U(x_i, t_j)}{\Delta t}\right] + O(\Delta t), \nonumber \\
		&=& \dfrac{\Delta t^{-\alpha}}{\Gamma (2 - \alpha)} \displaystyle\sum_{j=0}^{k}
		\left[(k+1-j)^{1-\alpha} - (k-j)^{1-\alpha}\right] \left(U(x_i, t_{j+1}) - U(x_i, t_j)\right) \label{TimeFracCaputo}\\
		&&+ O(\Delta t). \nonumber
		\end{eqnarray}
	Substituting the difference scheme (\ref{TimeFracCaputo}) into the proposed model equation (\ref{TimeFractionalModel}) and the initial - boundary conditions (\ref{InitialCondition}) - (\ref{BoundaryConditions}), we obtain: 
\begin{eqnarray}
	&&\displaystyle\sum_{j=0}^{k}
	\left[(k+1-j)^{1-\alpha} - (k-j)^{1-\alpha}\right] \left(U_i^{j+1} - U_i^j\right) + O(\Delta t + \Delta x) \nonumber \\
	&=& - \dfrac{\Delta t^\alpha\Gamma (2 - \alpha)}{\Delta x}\left[f(U_i^k, U_{i+1}^k) - f(U_{i-1}^k, U_i^k)\right] + \Delta t^\alpha\Gamma (2 - \alpha) S(U_i^k),
	\label{TimefractionalDiscritization}
\end{eqnarray} 
where according to \cite{toro2013riemann}, 
\begin{equation}
	f(U_i^k, U_{i+1}^k) = \dfrac{1}{2} \left(f(U_i^k) + f(U_{i+1}^k)\right) - \dfrac{1}{2} \left(B(U_{i+1/2}) \left(U_{i+1} - U_i \right)\right),
	\label{FluxChangeAtBoundary3}
\end{equation} 
is the Roe's numerical flux and $B(U_{i+1/2})$ refers to the averaged Jacobian matrix of $f(U).$ In order to dissolve discontinuities at the boundaries, an entropy condition is applied to Roe decomposition. For details the reader is encouraged to read \cite{nanyondo2024analysis} and references therein. The initial-boundary conditions are evaluated as follows
\begin{equation}
	U_i^0 = g_1(i \Delta x),~ U_0^k = g_2(k \Delta t),~ U_M^k = g_3(k \Delta t),~ i = 0, 1, \ldots, M,~ k = 0, 1, \ldots, N. \label{InitialBoundaryConditionsDiscritzed}
\end{equation} 

\begin{theorem}
	The explicit scheme (\ref{TimefractionalDiscritization}) with the initial - boundary conditions (\ref{InitialBoundaryConditionsDiscritzed}) is consistent.
\end{theorem}

\begin{proof}
	For $k=0,$ (\ref{TimefractionalDiscritization}) yields 
	\begin{equation}
		U_i^1 = U_i^0 - \dfrac{\Delta t^\alpha\Gamma (2 - \alpha)}{\Delta x}\left[f(U_i^0, U_{i+1}^0) - f(U_{i-1}^0, U_i^0)\right] + \Delta t^\alpha\Gamma (2 - \alpha) S(U_i^0),
		\label{TimefractionalDiscritizationj=k=0}
	\end{equation}
	Substitution of $\alpha = 1,$ into (\ref{TimefractionalDiscritizationj=k=0}) gives 
	\begin{equation}
		U_i^1 = U_i^0 - \dfrac{\Delta t}{\Delta x}\left[f(U_i^0, U_{i+1}^0) - f(U_{i-1}^0, U_i^0)\right] + \Delta t S(U_i^0).
	\end{equation}  
	Next, $j=k,$ is substituted into (\ref{TimefractionalDiscritization}) so that 
	\begin{eqnarray}
		U_i^{k+1} = U_i^k - \displaystyle\sum_{j=0}^{k-1}
		\left[(k+1-j)^{1-\alpha} - (k-j)^{1-\alpha}\right] \left(U_i^{j+1} - U_i^j\right) && \nonumber \\
		- \dfrac{\Delta t^\alpha\Gamma (2 - \alpha)}{\Delta x}\left[f(U_i^k, U_{i+1}^k) - f(U_{i-1}^k, U_i^k)\right] + \Delta t^\alpha\Gamma (2 - \alpha) S(U_i^k).
		\label{TimefractionalDiscritizationj=k}
	\end{eqnarray}
	Since $\alpha = 1,$ (\ref{TimefractionalDiscritizationj=k}) becomes
	\begin{equation}
		U_i^{k+1} = U_i^k - \dfrac{\Delta t}{\Delta x}\left[f(U_i^k, U_{i+1}^k) - f(U_{i-1}^k, U_i^k)\right] + \Delta t S(U_i^k).
	\end{equation} Hence, the explicit scheme (\ref{TimefractionalDiscritization}) is consistent.
\end{proof}

\subsection{Stability analysis of the finite difference scheme}\label{TimeFracStability} 
Let $\overline{U}_i^k,~U_i^k,$ for $i = 1, 2, \ldots, m-1,~ j = 1, 2, \ldots, n-1$ be the approximate and exact solutions, respectively of the difference schemes (\ref{TimefractionalDiscritization}) and (\ref{InitialBoundaryConditionsDiscritzed}). Further, let $\varepsilon_i^k = \overline{U}_i^ k - U_i^k$ be the numerical errors of the two solutions. If $k=0,$ based on (\ref{TimefractionalDiscritizationj=k=0}) we obtain
\begin{equation}
	\varepsilon_i^{1} = \varepsilon_i^0 - \dfrac{\Delta t^\alpha\Gamma(2-\alpha)}{\Delta x} \left[ f(\varepsilon_i^0, \varepsilon_{i+1}^0) - f(\varepsilon_{i-1}^0, \varepsilon_i^0)\right],
\end{equation} where $i = 1, 2, \ldots, M-1.$ If $k \geq 1,$  based on (\ref{TimefractionalDiscritizationj=k}) it follows that
\begin{equation}
	\varepsilon_i^{k+1} = \varepsilon_i^k - \displaystyle\sum_{j=0}^{k-1} \left[b_{j+1} - b_j\right]\left(\varepsilon_i^{j+1} - \varepsilon_i^j\right) - r \left[ f(\varepsilon_i^k, \varepsilon_{i+1}^k) - f(\varepsilon_{i-1}^k, \varepsilon_i^k)\right], 
\end{equation}
where $b_{j+1} = (k-j+1)^{1-\alpha},~ b_j = (k-j)^{1-\alpha},~ r = \dfrac{\Delta t^\alpha\Gamma(2-\alpha)}{\Delta x},~ k = 1, \ldots, N-1,~ i = 1, 2, \ldots, M-1.$ Let $E^k = \left(\varepsilon_1^k, \varepsilon_2^k, \dots, 	\varepsilon_{M-1}^k\right)^T$ be the error vector and $\mid \varepsilon_L^k \mid = \max\limits_{1 \leq i \leq M-1} \mid \varepsilon_i^k \mid.$ Then the following Theorem is obtained by mathematical induction.
\begin{theorem}
	For any $k \geq 0,$ $\parallel E^k \parallel_\infty \leq \parallel E^0 \parallel_\infty.$ That is, the explicit difference schemes (\ref{TimefractionalDiscritization}) and (\ref{InitialBoundaryConditionsDiscritzed}) are stable. 
\end{theorem}
\begin{proof}
	Since $b_{j+1} - b_j < 0,$ when $k=0$
	\begin{eqnarray*}
		\parallel E^1 \parallel_\infty &=& \mid \varepsilon_L^1 \mid \\
		&=& \mid \varepsilon_i^0 - \dfrac{\Delta t^\alpha \Gamma(2-\alpha)}{\Delta x}\left[ f(\varepsilon_{i+1/2}^0) - f(\varepsilon_{i-1/2}^0)\right] \mid\\
		&\leq& \mid \varepsilon_i^0 \mid - \dfrac{\Delta t^\alpha \Gamma(2-\alpha)}{\Delta x} \mid f(\varepsilon_{i+1/2}^0) - f(\varepsilon_{i-1/2}^0) \mid\\
		&\leq& \mid \varepsilon_i^0 \mid \\
		&\leq& \parallel E^0 \parallel_\infty
	\end{eqnarray*}
	Next, when $k \geq 1,$ suppose $\parallel E^j \parallel_\infty \leq \parallel E^0 \parallel_\infty,~ j = 1, 2, \ldots, k.$ Then
	\begin{eqnarray*}
		\parallel E^{k+1} \parallel_\infty &=& \mid \varepsilon_L^{k+1} \mid \\
		&=& \mid \varepsilon_i^k - \sum_{j=0}^{k-1}\left(b_{j+1} - b_j\right)\left(\varepsilon_i^{j+1} - \varepsilon_i^j \right) - r \left[ f(\varepsilon_{i+1/2}^k) - f(\varepsilon_{i-1/2}^k)\right] \mid \\
		&\leq& \mid \varepsilon_i^k \mid - \sum_{j=0}^{k-1}\left(b_{j+1} - b_j\right)\left(\mid \varepsilon_i^{j+1} \mid - \mid \varepsilon_i^{j} \mid \right) - r \mid f(\varepsilon_{i+1/2}^k) - f(\varepsilon_{i-1/2}^k) \mid \\ 
		&\leq& \mid \varepsilon_i^k \mid - \left[ \sum_{j=0}^{k-1}\left(b_{j+1} - b_j\right)\left(\mid \varepsilon_i^{j+1} \mid - \mid \varepsilon_i^{j} \mid \right) + r \mid f(\varepsilon_{i+1/2}^k) - f(\varepsilon_{i-1/2}^k) \mid \right] \\ 
		&\leq& \mid \varepsilon_i^k \mid \\
		&\leq& \parallel E^0 \parallel_\infty.
	\end{eqnarray*} So the difference scheme is stable.
\end{proof} 

\subsection{Convergence analysis of the scheme}\label{TimeFracConvergence} 
Suppose $U(x_i, t_k),~ i = 1, 2, \ldots, M-1,~ k = 1, 2, \ldots, N-1$ is the exact solution of the differential equations (\ref{TimeFractionalModel}) and (\ref{BoundaryConditions}) at grid point $(x_i, t_k)$. Define the errors between exact and numerical solutions as 
$\eta_i^k = U(x_i, t_k) - U_i^k,~ i= 1, 2,\ldots$ and denote $Y^k = \left(\eta_1^k, \eta_2^k, \ldots, \eta_{M-1}^k\right)^T.$ Substituting the difference equation and using $\eta^0 = 0,$ we derive the error iterative scheme as follows. When $k = 1,$ 
\begin{equation}
	\eta_i^1 = R_{i,1},~ i=1, 2, \ldots, M-1.
\end{equation}
When $k\geq 1,$
\begin{equation}
	\eta_i^{k+1} = \eta_i^k - \sum_{j=0}^{k-1}\left(b_{j+1} - b_j\right) \left(\eta_i^{j+1} - \eta_i^j\right) - r\left[ f(\eta_{i+1/2}^k) - f(\eta_{i-1/2}^k)\right] + R_{i,k+1},
\end{equation}
where we use the solution $\mid R_{i,k+1} \mid \leq C\left(\Delta t^{1+\alpha} + \Delta t^\alpha \Delta x\right),~ i=1,2, \ldots, M-1,~ k=1,2,3, \ldots, N-1.$ Based on the mathematical induction, we obtain the following theorem. 
\begin{theorem}
For any $k \geq 1,$ 
	\begin{equation}
		\parallel Y^k \parallel_\infty \leq b_{k-1}^{-1}C\left(\Delta t^{1+\alpha} + \Delta t^\alpha \Delta x\right),~ k=1,2,\ldots,N.
	\end{equation}
\end{theorem}
\begin{proof}
	Let $\mid \eta_L^k \mid = \max\limits_{1 \leq i \leq M-1} \mid \eta_i^k \mid,~ k=1,2,\ldots,N,$ such that 
	\begin{eqnarray*}
		\parallel Y^1 \parallel_\infty &=& \mid \eta_i^1 \mid, \\
		&=& \mid R_{i,1} \mid, \\
		&\leq& C\left(\Delta t^{1+\alpha} + \Delta t^\alpha \Delta x\right), \\
		&\leq& b_0^{-1}C\left(\Delta t^{1+\alpha} + \Delta t^\alpha \Delta x\right).
	\end{eqnarray*}
	Suppose $\parallel Y^j \parallel_\infty \leq b_{j-1}^{-1}C\left(\Delta t^{1+\alpha} + \Delta t^\alpha \Delta x\right),~ j=1,2,\ldots,k,~ \text{and}~ b_k^{-1} \geq b_j^{-1}~ (j=0,1,2,\ldots,k).$ We then obtain
	\begin{eqnarray*}
		\parallel Y^{k+1} \parallel_\infty &=& \mid \eta_L^{k+1} \mid, \\
		&=& \mid \eta_i^k - \sum_{j=0}^{k-1}\left(b_{j+1} - b_j\right) \left(\eta_i^{j+1} - \eta_i^j \right) 
		- r\left( f(\eta_{i+1/2}^k) - f(\eta_{i-1/2}^k)\right) + R_{i,k+1} \mid,\\
		&\leq& \mid \eta_i^k \mid - \mid \sum_{j=0}^{k-1}\left(b_{j+1} - b_j\right) \left(\eta_i^{j+1} - \eta_i^j \right) 
		+ r\left( f(\eta_{i+1/2}^k) - f(\eta_{i-1/2}^k)\right) \mid + \mid R_{i,k+1} \mid,\\
		&\leq& \mid \eta_i^k \mid + \mid R_{i,k+1}\mid, \\
		&\leq& \parallel Y^k \parallel_\infty  + C\left(\Delta t^{1+\alpha} + \Delta t^\alpha \Delta x\right), \\
		&\leq& b_{k-1}^{-1}C\left(\Delta t^{1+\alpha} + \Delta t^\alpha \Delta x\right) + b_0^{-1}C\left(\Delta t^{1+\alpha} + \Delta t^\alpha \Delta x\right), ~\text{since}~ b_0 ~ \text{is} ~ 1, \\
		&\leq& (b_{k-1}^{-1} + b_{0}^{-1})C\left(\Delta t^{1+\alpha} + \Delta t^\alpha \Delta x\right), \\
		&\leq& b_{k}^{-1}C\left(\Delta t^{1+\alpha} + \Delta t^\alpha \Delta x\right),
	\end{eqnarray*}
	where $b_{k-1}^{-1} + b_{0}^{-1}$ is assumed to be less than or equal to $b_k^{-1}$ 
\end{proof}
Because 
\begin{eqnarray}
	\lim\limits_{k\rightarrow \infty} \dfrac{b_k^{-1}}{k^\alpha} &=& \lim\limits_{k\rightarrow \infty} \dfrac{k^{-\alpha}}{(k+1)^{1-\alpha} - k^{1-\alpha}}, ~ \text{substituting}~ b_k \nonumber \\
	&=& \lim\limits_{k\rightarrow \infty} \dfrac{k^{-\alpha}}{  k^{1-\alpha} \left[(1+1/k)^{1-\alpha} - 1 \right] }, ~ \text{factoring}~ k, \nonumber \\
	&=& \lim\limits_{k\rightarrow \infty} \dfrac{k^{-1}}{  (1+1/k)^{1-\alpha} - 1 }, ~ \text{dividing through by}~ k^{1-\alpha}, \nonumber \\
	&=& \lim\limits_{k\rightarrow \infty} \dfrac{k^{-1}}{ \left[1 + (1-\alpha)(1/k) \right] - 1 }, ~ \text{binomial expansion}, \nonumber \\
	&=& \dfrac{1}{1 - \alpha}, \label{bk}
\end{eqnarray}
there exists a constant $C_1 = \dfrac{1}{1 - \alpha} >0$ such that 
\begin{equation}
	b_k^{-1} \leq C_1 k^{\alpha},
\end{equation} 
by making $b_k^{-1}$ of (\ref{bk}), the subject.
Since $k \Delta t < T$ is bounded we deduce 
\begin{eqnarray*}
	\mid U_i^k - U(x_i,t_k) \mid & \leq & \mid \eta_L^k \mid, \\
	& \leq & b_k^{-1}C \Delta t^\alpha \left(\Delta t + \Delta x\right), \\
	& \leq & C_1 k^{\alpha} C \Delta t^\alpha \left(\Delta t + \Delta x\right), ~\text{substituting for} ~ b_k^{-1} \\
	&\leq & CC_1 \left(k \Delta t\right)^{\alpha} \left(\Delta t + \Delta x\right), \\
	& \leq & \hat{C} \left(\Delta t + \Delta x\right), 
\end{eqnarray*}
where $ \hat{C} = CC_1 \left(k \Delta t\right)^{\alpha}.$ Hence, the scheme is convergent.

\section{Numerical results}
Under this section, numerical solutions of
the proposed model (\ref{TimeFractionalModel}), with initial condition (\ref{InitialCondition}) and boundary conditions (\ref{BoundaryConditions}), for varying values of
the fractional derivative $\alpha = 1, 0.9, 0.8, 0.7$ are presented. We describe vehicles’ density and velocity distributions when motorcycles’ proportion is set to 20\% and 90\%. Roe’s numerical scheme is utilized in which periodic boundary conditions are applied in order
to represent a circular road. In an attempt to dissolve discontinuities at boundaries, an entropy fix is applied
to the Roe’s numerical scheme. Consistency, stability and convergence are established. Table \ref{ParameterValues} shows parameter values that are used for simulations. Note that when $\delta$ takes on values $0$ and $1,$ results become unstable (undefined). Figures \ref{FreewayDelta90One} - \ref{CongestedDelta20Two} display simulation results of the proposed model.          

\vspace{-0.1cm}

\begin{table}[H]
	\centering
	\caption{Simulation parameters.}
	\begin{tabular}{l|c|r}
		\hline
		\textbf{Description} & \textbf{Value} & Source\\ \hline
		Length of the road & 500m & \\ 
		Road step, h & 5 m & \cite{khan2019macroscopic}\\
		Time step & 0.05 s &  \\
		Relaxation time,$\tau_m,~ \tau_c$ & 3 s, 5 s & \cite{mohan2021multi}\\
		Equilibrium velocity, $v_{em}, v_{ec}$ & Greenshields & \cite{khan2019macroscopic}\\
		Maximum normalized density & $\rho=\rho_m+\rho_c = 1$ & \cite{khan2019macroscopic}\\
		Maximum speed, $v_{m\max}$ & 11m/s & \cite{mohan2017HTModel} \\
		Maximum speed, $v_{c\max}$ & 13.8m/s & \cite{Ministryofworksandtransport} \\
		Maximum area occupancy, $AO_m^{\max}$ & 0.85 & \cite{mohan2013heterogeneous,mohan2017HTModel} \\
		Maximum area occupancy, $AO_c^{\max}$ & 0.74 & \cite{mohan2013heterogeneous,mohan2017HTModel} \\		
		$\gamma_i,~ i=m,c$ & 2.23, 2.12 & \cite{mohan2013heterogeneous,mohan2017HTModel}\\
		Width of the road & 12m & \cite{Ministryofworksandtransport} \\
		Width of a car & 1.6m & \cite{arasan2008measuring,mohan2021multi}\\
		Vehicle class proportion, $\delta$ & 20\%, 90\% & \\
		Length of a car & 4m & \cite{arasan2008measuring,mohan2021multi}\\
		Length of a motorcycle & 1.8m & \cite{arasan2008measuring}\\
		Simulation time, T & 60 seconds & \cite{khan2019macroscopic} \\ \hline
	\end{tabular}
	\label{ParameterValues}
\end{table}

\subsection{Freeway traffic on a roundabout}
In his subsection, we consider a free traffic flow with motorcycles proportion $\delta$ initially fixed at $90\%$ for $\alpha = 1, 0.9, 0.8, 0.7$ and $0s, 1s, 20s, 40s$ and $60s.$ The initial total density of traffic flow denoted $\rho_0,$ is set to be 
\begin{eqnarray}
	\rho_0 &=&
	\begin{cases}
		0.1, &\text{for}~ x < 100 \\
		0.2,& \text{for}~ x \geq 100.
	\end{cases}
\end{eqnarray}
Since $\delta = 90\%,$ it follows that $\rho_m = 0.9 \rho_0$ and $\rho_c = 0.1 \rho_0.$ The initial density and velocity distributions of the respective vehicle classes are plotted in Figures \ref{FreewayDensitymDelta90_0s} -  \ref{FreewayVelocitycDelta90_0s}. A shock wave develops at 1s for both vehicle classes and becomes smooth thereafter (see Figures \ref{FreewayDensitymDelta90_20s} - \ref{FreewayVelocitycDelta90_60s}). It is observed from Figures \ref{FreewayDelta90One} and \ref{FreewayDelta90Two} that values for integer order model fluctuate so much while those for fractional order model are moderated as time progresses. 

\begin{figure}[H]
	\vspace{0pt}
	\centering
	\begin{subfigure}[t]{0.32\textwidth}
		\vspace{0pt}
		\includegraphics[width=\linewidth]{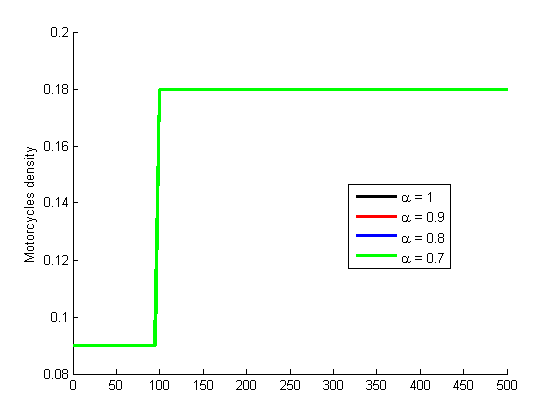}
		\caption{$\rho_m$ vs x at $T=0s$}\label{FreewayDensitymDelta90_0s}
	\end{subfigure}
	\begin{subfigure}[t]{0.32\textwidth}
		\vspace{0pt}
		\includegraphics[width=\linewidth]{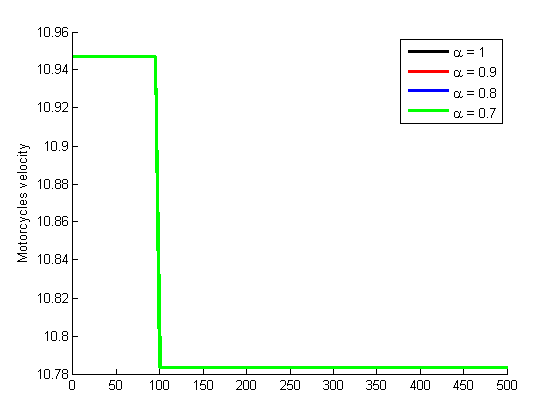}
		\caption{$v_m$ vs x at $T=0s$}\label{FreewayVelocitymDelta90_0s}
	\end{subfigure}
	\begin{subfigure}[t]{0.32\textwidth}
		\vspace{0pt}
		\includegraphics[width=\linewidth]{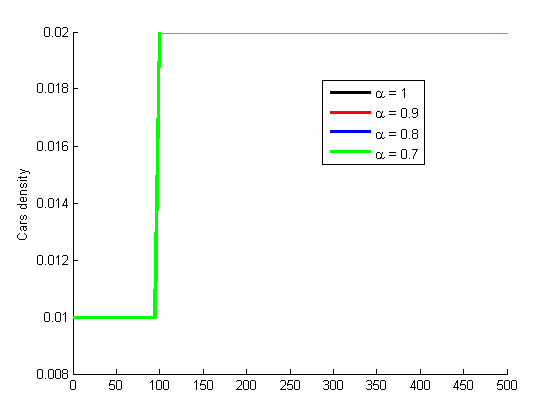}
		\caption{$\rho_c$ vs x at $T=0s$}\label{FreewayDensitycDelta90_0s}
	\end{subfigure}
	\begin{subfigure}[t]{0.32\textwidth}
		\vspace{0pt}
		\includegraphics[width=\linewidth]{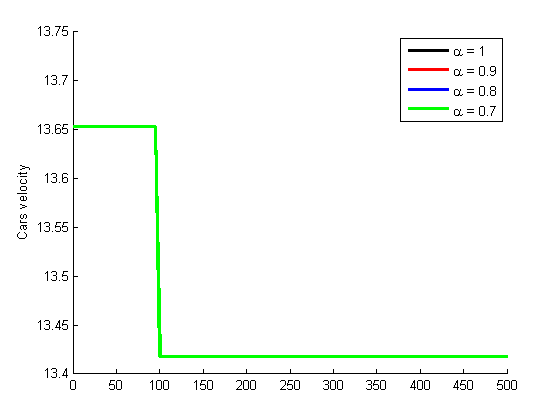}
		\caption{$v_c$ vs x at $T=0s$}\label{FreewayVelocitycDelta90_0s}
	\end{subfigure}
	\begin{subfigure}[t]{0.32\textwidth}
		\vspace{0pt}
		\includegraphics[width=\linewidth]{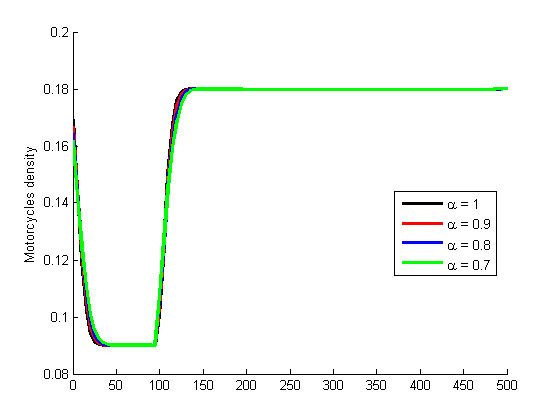}
		\caption{$\rho_m$ vs x at $T=1s$}\label{FreewayDensitymDelta90_1s}
	\end{subfigure}
	\begin{subfigure}[t]{0.32\textwidth}
		\vspace{0pt}
		\includegraphics[width=\linewidth]{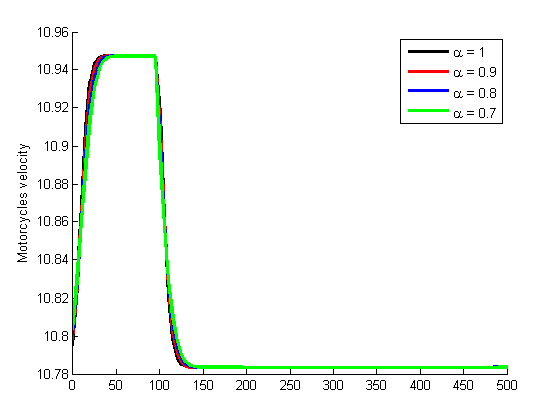}
		\caption{$v_m$ vs x at $T=1s$}\label{FreewayVelocitymDelta90_1s}
	\end{subfigure}
	\begin{subfigure}[t]{0.32\textwidth}
		\vspace{0pt}
		\includegraphics[width=\linewidth]{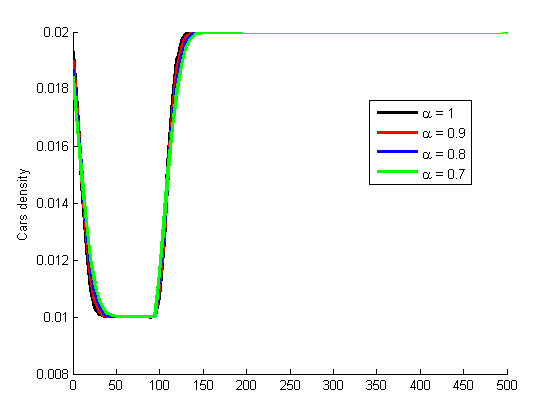}
		\caption{$\rho_c$ vs x at $T=1s$}\label{FreewayDensitycDelta90_1s}
	\end{subfigure}
	\begin{subfigure}[t]{0.32\textwidth}
		\vspace{0pt}
		\includegraphics[width=\linewidth]{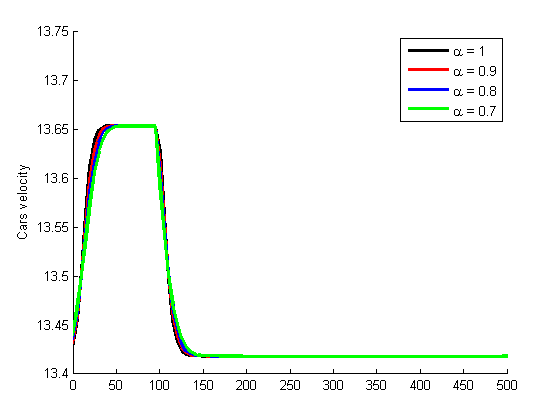}
		\caption{$v_c$ vs x at $T=1s$}\label{FreewayVelocitycDelta90_1s}
	\end{subfigure}
	\begin{subfigure}[t]{0.32\textwidth}
		\vspace{0pt}
		\includegraphics[width=\linewidth]{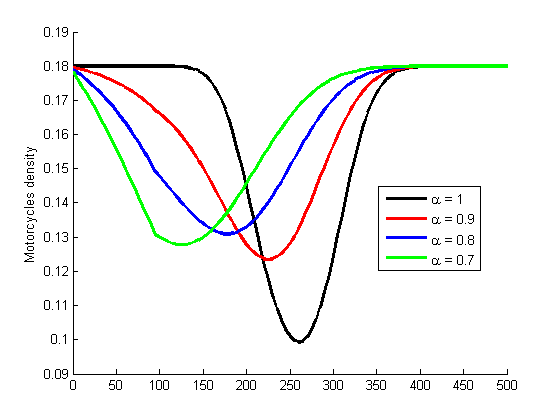}
		\caption{$\rho_m$ vs x at $T=20s$}\label{FreewayDensitymDelta90_20s}
	\end{subfigure}
	\begin{subfigure}[t]{0.32\textwidth}
		\vspace{0pt}
		\includegraphics[width=\linewidth]{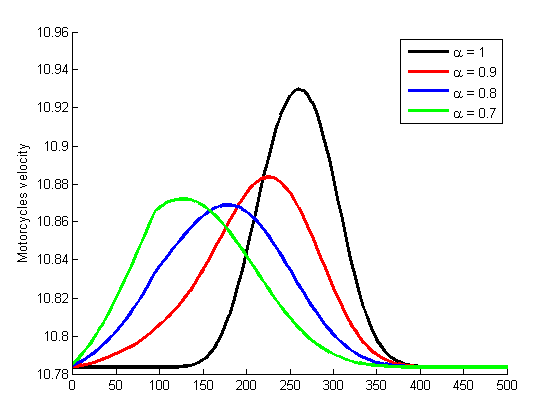}
		\caption{$v_m$ vs x at $T=20s$}\label{FreewayVelocitymDelta90_20s}
	\end{subfigure}
	\begin{subfigure}[t]{0.32\textwidth}
		\vspace{0pt}
		\includegraphics[width=\linewidth]{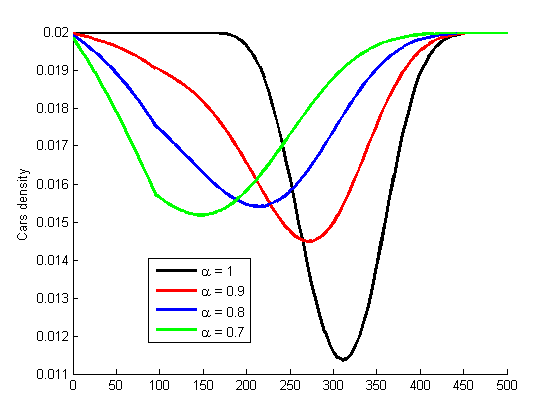}
		\caption{$\rho_c$ vs x at $T=20s$}\label{FreewayDensitycDelta90_20s}
	\end{subfigure}
	\begin{subfigure}[t]{0.32\textwidth}
		\vspace{0pt}
		\includegraphics[width=\linewidth]{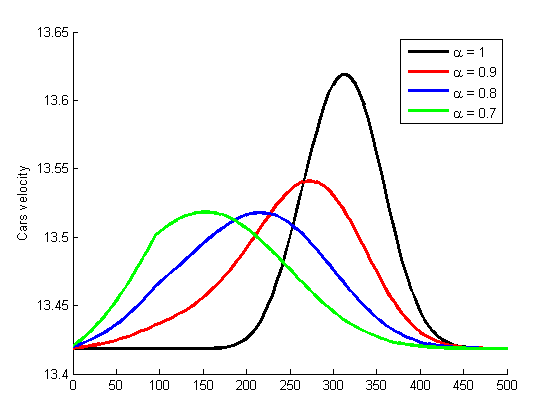}
		\caption{$v_c$ vs x at $T=20s$}\label{FreewayVelocitycDelta90_20s}
	\end{subfigure}
	\caption{Densities and velocities of the proposed model on a $500 m$ freeway circular road when $ \alpha = 1, 0.9, 0.8, 0.7,~ \delta = 90\%,$ at $T = 0s, 1s, 20s.$} \label{FreewayDelta90One} 
	\vspace{-0.1cm}
\end{figure}

\begin{figure}[H]
	\vspace{0pt}
	\centering
	\begin{subfigure}[t]{0.32\textwidth}
		\vspace{0pt}
		\includegraphics[width=\linewidth]{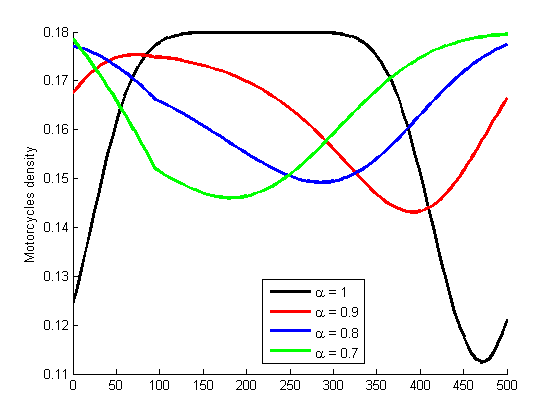}
		\caption{$\rho_m$ vs x at $T=40s$}\label{FreewayDensitymDelta90_40s}
	\end{subfigure}
	\begin{subfigure}[t]{0.32\textwidth}
		\vspace{0pt}
		\includegraphics[width=\linewidth]{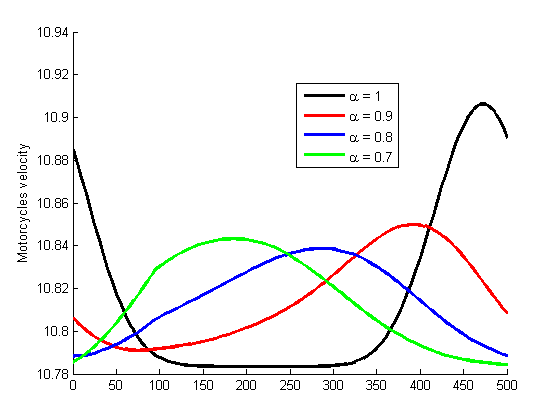}
		\caption{$v_m$ vs x at $T=40s$}\label{FreewayVelocitymDelta90_40s}
	\end{subfigure}
	\begin{subfigure}[t]{0.32\textwidth}
		\vspace{0pt}
		\includegraphics[width=\linewidth]{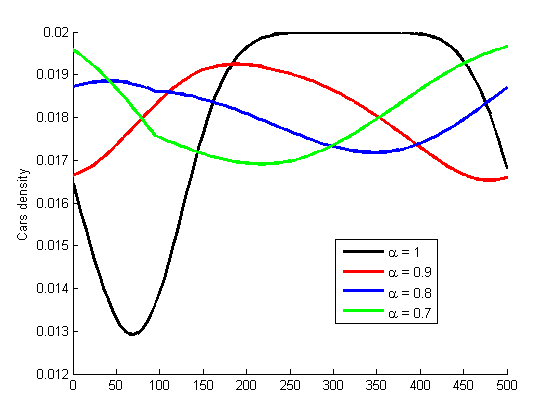}
		\caption{$\rho_c$ vs x at $T=40s$}\label{FreewayDensitycDelta90_40s}
	\end{subfigure}
	\begin{subfigure}[t]{0.32\textwidth}
		\vspace{0pt}
		\includegraphics[width=\linewidth]{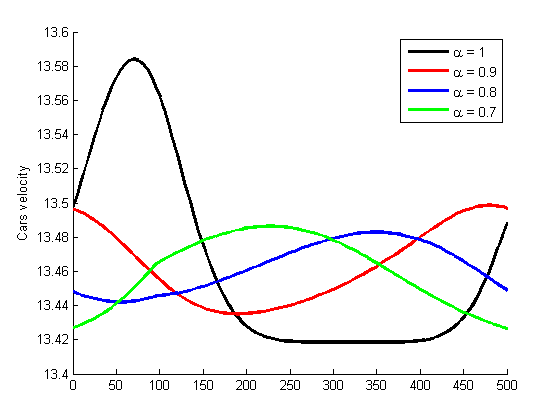}
		\caption{$v_c$ vs x at $T=40s$}\label{FreewayVelocitycDelta90_40s}
	\end{subfigure}
	\begin{subfigure}[t]{0.32\textwidth}
		\vspace{0pt}
		\includegraphics[width=\linewidth]{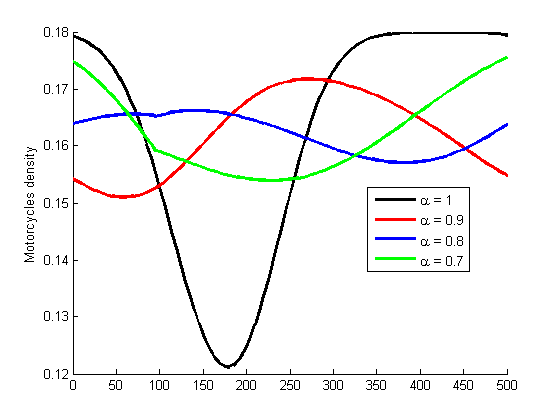}
		\caption{$\rho_m$ vs x at $T=60s$}\label{FreewayDensitymDelta90_60s}
	\end{subfigure}
	\begin{subfigure}[t]{0.32\textwidth}
		\vspace{0pt}
		\includegraphics[width=\linewidth]{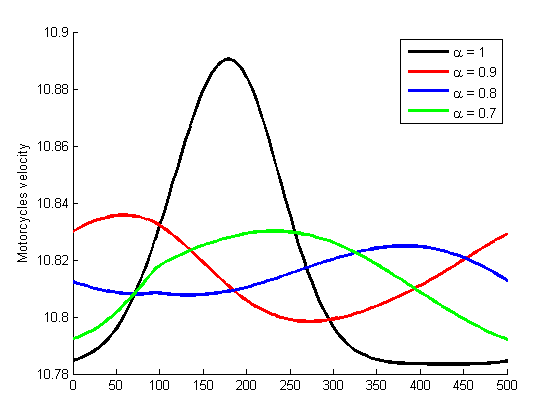}
		\caption{$v_m$ vs x at $T=60s$}\label{FreewayVelocitymDelta90_60s}
	\end{subfigure}
	\begin{subfigure}[t]{0.32\textwidth}
		\vspace{0pt}
		\includegraphics[width=\linewidth]{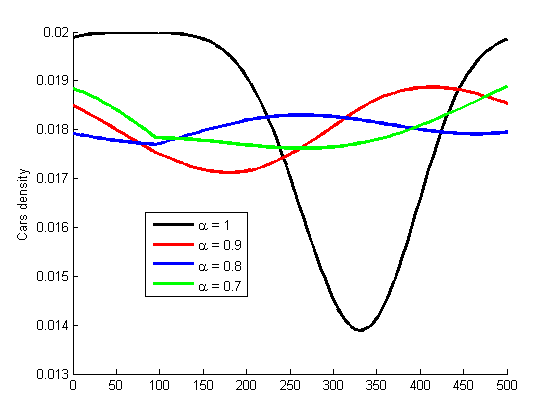}
		\caption{$\rho_c$ vs x at $T=60s$}\label{FreewayDensitycDelta90_60s}
	\end{subfigure}
	\begin{subfigure}[t]{0.32\textwidth}
		\vspace{0pt}
		\includegraphics[width=\linewidth]{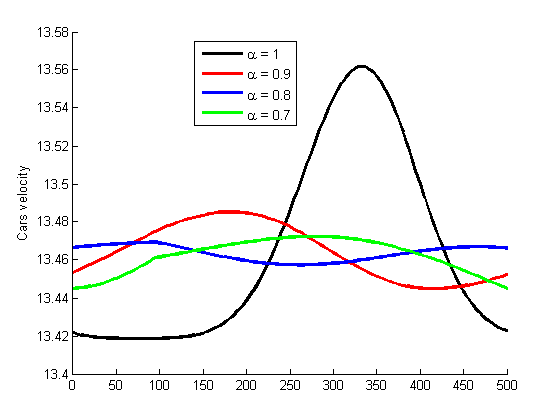}
		\caption{$v_c$ vs x at $T=60s$}\label{FreewayVelocitycDelta90_60s}
	\end{subfigure}
	\caption{Densities and velocities of the proposed model on a $500 m$ freeway circular road when $ \alpha = 1, 0.9, 0.8, 0.7,~ \delta = 90\%,$ at $T = 40s, 60s.$} \label{FreewayDelta90Two} 
	\vspace{-0.1cm}
\end{figure}

Next, $\delta$ is set to $20\%,$ such that $\rho_m = 0.2 \rho_0$ and $\rho_c = 0.8 \rho_0,$ for motorcycles and cars, respectively. The plot of such initial density profile is shown in Figure \ref{FreewayDensitymDelta20_0s} and \ref{FreewayDensitycDelta20_0s} with corresponding velocities in Figures \ref{FreewayVelocitymDelta20_0s} and \ref{FreewayVelocitycDelta20_0s}.
Similar results as in Figures \ref{FreewayDelta90One} and \ref{FreewayDelta90Two} are observed when $\delta = 20\%$ (see Figures \ref{FreewayDensitymDelta20_1s} - \ref{FreewayVelocitycDelta20_60s}). 
Different from the previous case when $\delta = 90\%,$ both vehicle classes move with slightly lower velocities.

\begin{figure}[H]
	\vspace{0pt}
	\centering
	\begin{subfigure}[t]{0.32\textwidth}
		\vspace{0pt}
		\includegraphics[width=\linewidth]{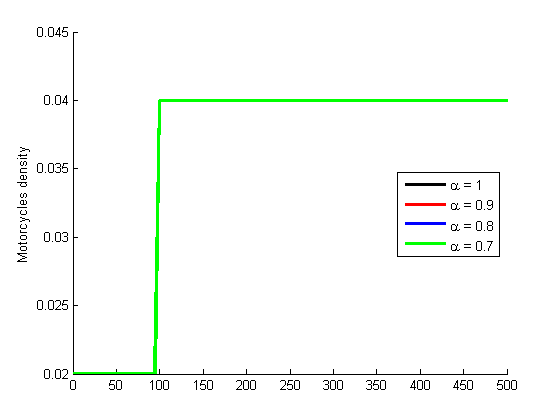}
		\caption{$\rho_m$ vs x at $T=0s$}\label{FreewayDensitymDelta20_0s}
	\end{subfigure}
	\begin{subfigure}[t]{0.32\textwidth}
		\vspace{0pt}
		\includegraphics[width=\linewidth]{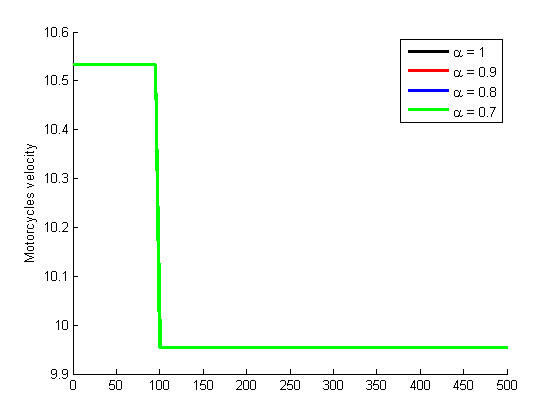}
		\caption{$v_m$ vs x at $T=0s$}\label{FreewayVelocitymDelta20_0s}
	\end{subfigure}
	\begin{subfigure}[t]{0.32\textwidth}
		\vspace{0pt}
		\includegraphics[width=\linewidth]{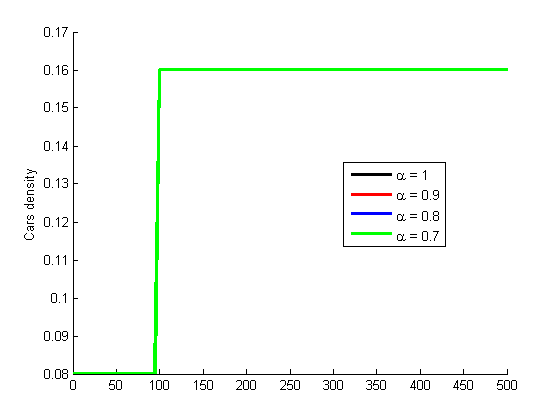}
		\caption{$\rho_c$ vs x at $T=0s$}\label{FreewayDensitycDelta20_0s}
	\end{subfigure}
	\begin{subfigure}[t]{0.32\textwidth}
		\vspace{0pt}
		\includegraphics[width=\linewidth]{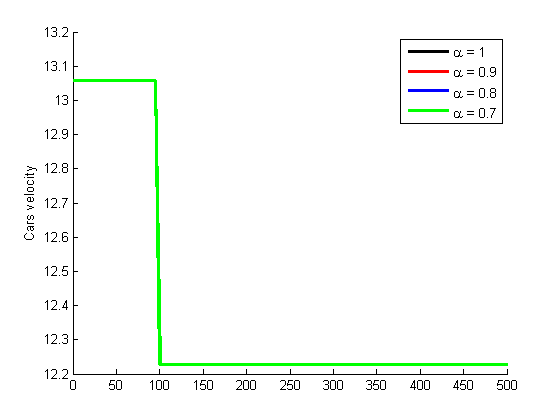}
		\caption{$v_c$ vs x at $T=0s$}\label{FreewayVelocitycDelta20_0s}
	\end{subfigure}
	\begin{subfigure}[t]{0.32\textwidth}
		\vspace{0pt}
		\includegraphics[width=\linewidth]{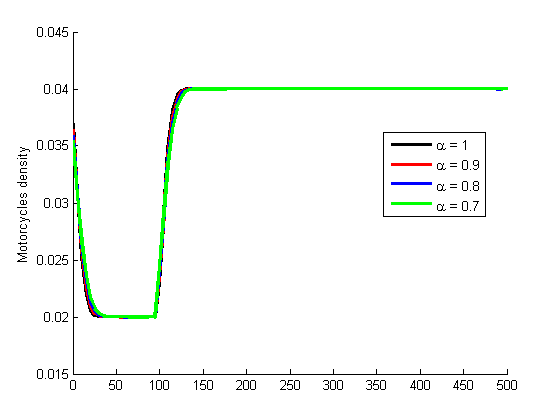}
		\caption{$\rho_m$ vs x at $T=1s$}\label{FreewayDensitymDelta20_1s}
	\end{subfigure}
	\begin{subfigure}[t]{0.32\textwidth}
		\vspace{0pt}
		\includegraphics[width=\linewidth]{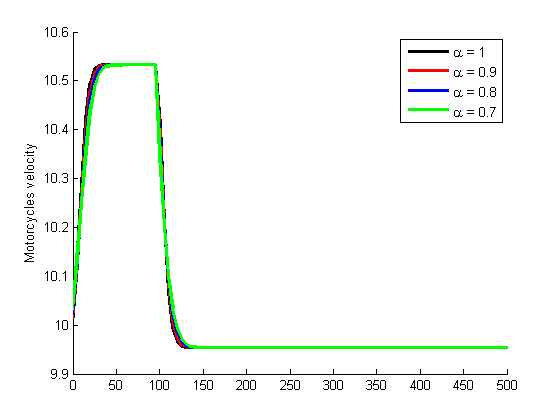}
		\caption{$v_m$ vs x at $T=1s$}\label{FreewayVelocitymDelta20_1s}
	\end{subfigure}
	\begin{subfigure}[t]{0.32\textwidth}
		\vspace{0pt}
		\includegraphics[width=\linewidth]{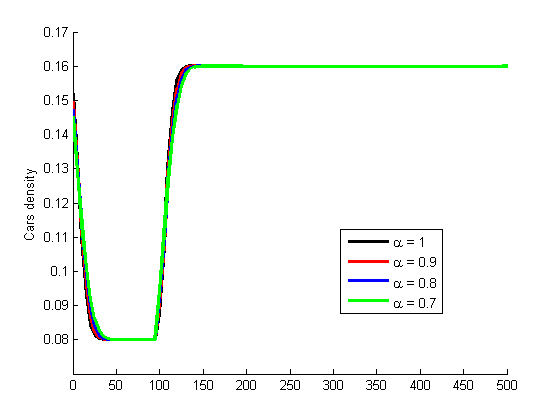}
		\caption{$\rho_c$ vs x at $T=1s$}\label{FreewayDensitycDelta20_1s}
	\end{subfigure}
	\begin{subfigure}[t]{0.32\textwidth}
		\vspace{0pt}
		\includegraphics[width=\linewidth]{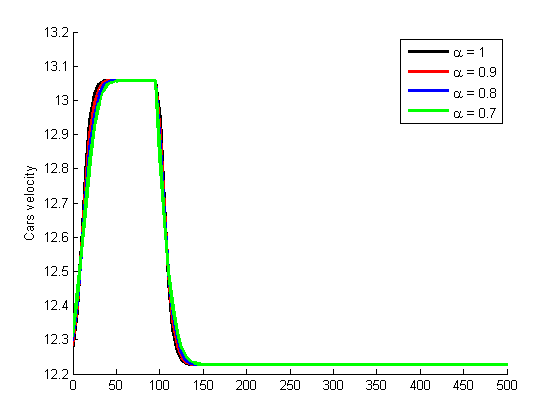}
		\caption{$v_c$ vs x at $T=1s$}\label{FreewayVelocitycDelta20_1s}
	\end{subfigure}
	\begin{subfigure}[t]{0.32\textwidth}
		\vspace{0pt}
		\includegraphics[width=\linewidth]{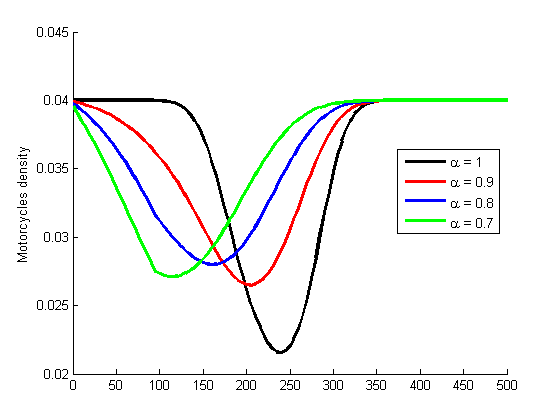}
		\caption{$\rho_m$ vs x at $T=20s$}\label{FreewayDensitymDelta20_20s}
	\end{subfigure}
	\begin{subfigure}[t]{0.32\textwidth}
		\vspace{0pt}
		\includegraphics[width=\linewidth]{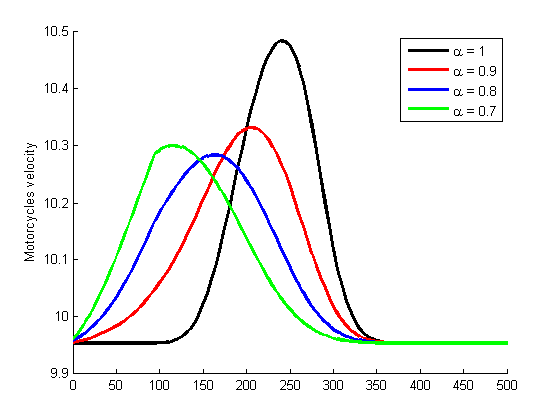}
		\caption{$v_m$ vs x at $T=20s$}\label{FreewayVelocitymDelta20_20s}
	\end{subfigure}
	\begin{subfigure}[t]{0.32\textwidth}
		\vspace{0pt}
		\includegraphics[width=\linewidth]{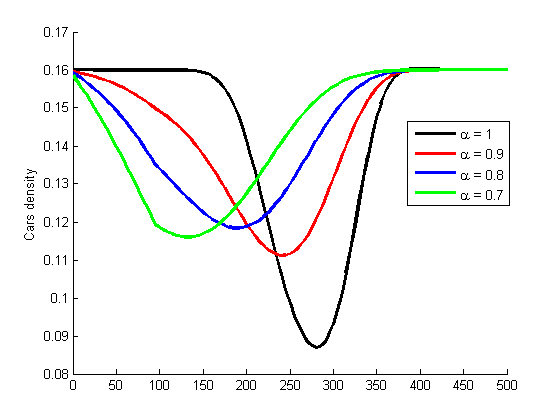}
		\caption{$\rho_c$ vs x at $T=20s$}\label{FreewayDensitycDelta20_20s}
	\end{subfigure}
	\begin{subfigure}[t]{0.32\textwidth}
		\vspace{0pt}
		\includegraphics[width=\linewidth]{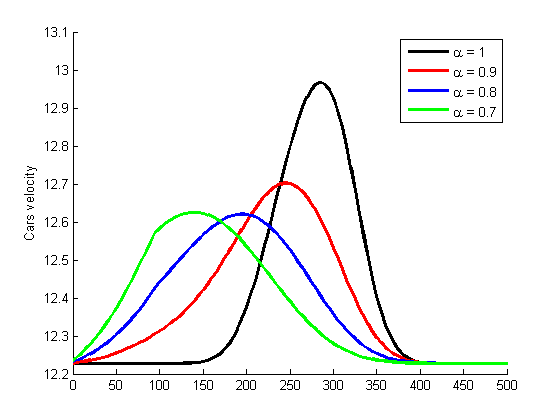}
		\caption{$v_c$ vs x at $T=20s$}\label{FreewayVelocitycDelta20_20s}
	\end{subfigure}
	\caption{Densities and velocities of the proposed model on a $500 m$ freeway circular road when $ \alpha = 1, 0.9, 0.8, 0.7,~ \delta = 20\%,$ at $T = 0s, 1 s, 20s.$} \label{FreewayDelta20One} 
	\vspace{-0.1cm}
\end{figure}

\begin{figure}[H]
	\vspace{0pt}
	\centering
	\begin{subfigure}[t]{0.32\textwidth}
		\vspace{0pt}
		\includegraphics[width=\linewidth]{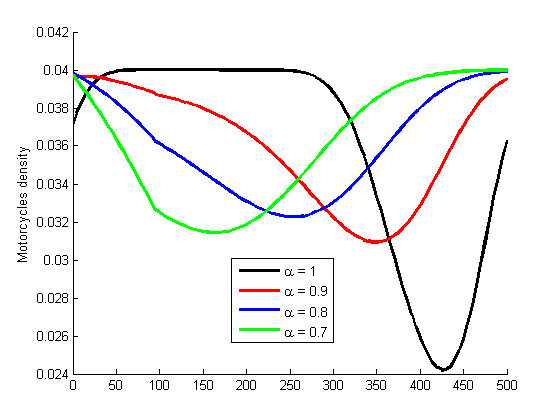}
		\caption{$\rho_m$ vs x at $T=40s$}\label{FreewayDensitymDelta20_40s}
	\end{subfigure}
	\begin{subfigure}[t]{0.32\textwidth}
		\vspace{0pt}
		\includegraphics[width=\linewidth]{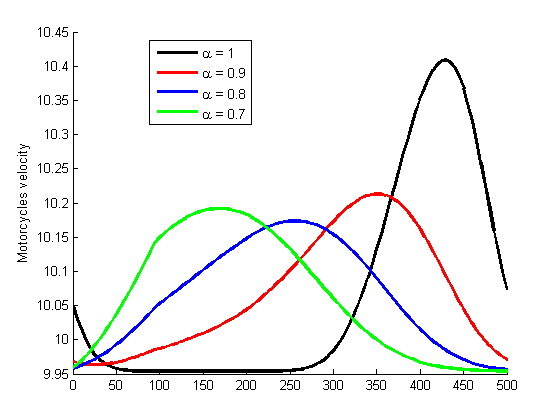}
		\caption{$v_m$ vs x at $T=40s$}\label{FreewayVelocitymDelta20_40s}
	\end{subfigure}
	\begin{subfigure}[t]{0.32\textwidth}
		\vspace{0pt}
		\includegraphics[width=\linewidth]{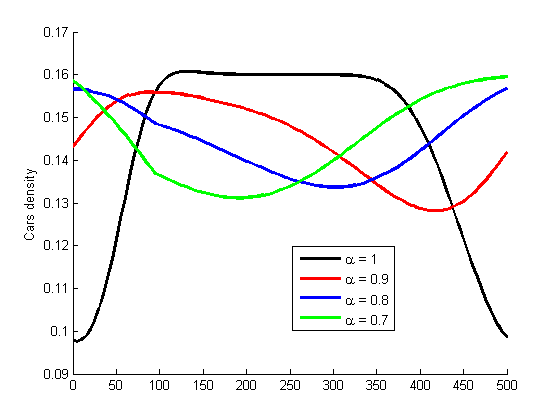}
		\caption{$\rho_c$ vs x at $T=40s$}\label{FreewayDensitycDelta20_40s}
	\end{subfigure}
	\begin{subfigure}[t]{0.32\textwidth}
		\vspace{0pt}
		\includegraphics[width=\linewidth]{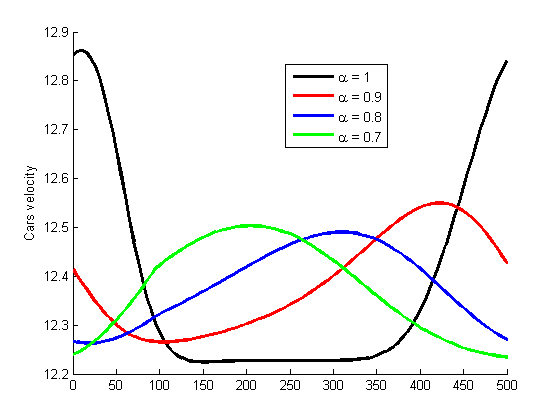}
		\caption{$v_c$ vs x at $T=40s$}\label{FreewayVelocitycDelta20_40s}
	\end{subfigure}
	\begin{subfigure}[t]{0.32\textwidth}
		\vspace{0pt}
		\includegraphics[width=\linewidth]{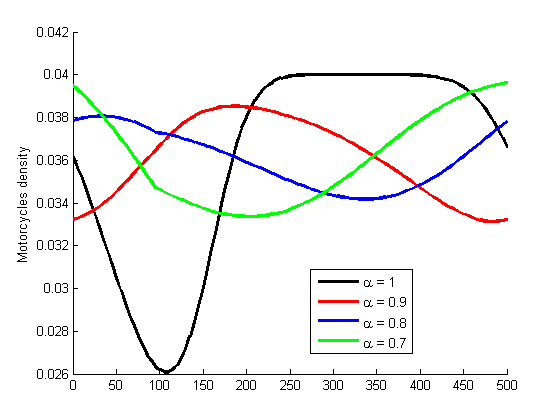}
		\caption{$\rho_m$ vs x at $T=60s$}\label{FreewayDensitymDelta20_60s}
	\end{subfigure}
	\begin{subfigure}[t]{0.32\textwidth}
		\vspace{0pt}
		\includegraphics[width=\linewidth]{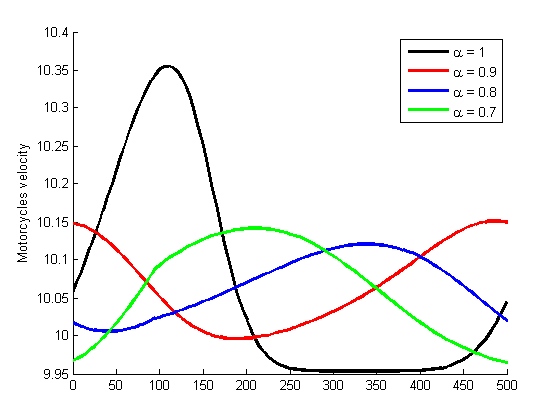}
		\caption{$v_m$ vs x at $T=60s$}\label{FreewayVelocitymDelta20_60s}
	\end{subfigure}
	\begin{subfigure}[t]{0.32\textwidth}
		\vspace{0pt}
		\includegraphics[width=\linewidth]{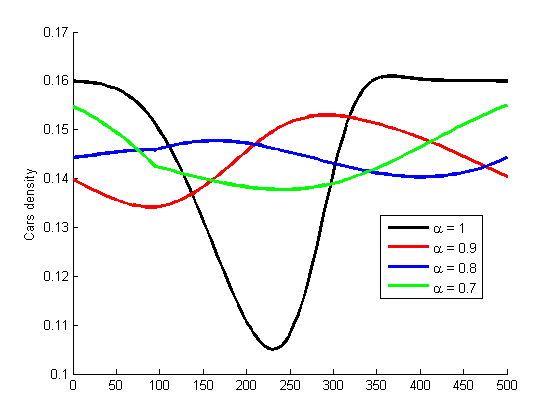}
		\caption{$\rho_c$ vs x at $T=60s$}\label{FreewayDensitycDelta20_60s}
	\end{subfigure}
	\begin{subfigure}[t]{0.32\textwidth}
		\vspace{0pt}
		\includegraphics[width=\linewidth]{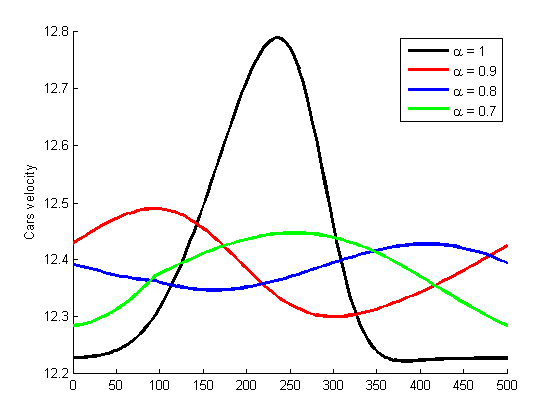}
		\caption{$v_c$ vs x at $T=60s$}\label{FreewayVelocitycDelta20_60s}
	\end{subfigure}
	\caption{Densities and velocities of the proposed model on a $500 m$ freeway circular road when $ \alpha = 1, 0.9, 0.8, 0.7,~ \delta = 20\%,$ at $T = 20s, 40s.$} \label{FreewayDelta20Two} 
	\vspace{-0.1cm}
\end{figure}
\subsection{Congested traffic on a roundabout}
Here we study a case for congested $500 m$ circular road with initial density $\rho_0,$ given by
\begin{eqnarray} 
	\rho_0 &=&
	\begin{cases}
		0.1, &\text{for}~ x \leq 130 \\
		0.8,& \text{for}~ 150 < x < 180 \\
		0.1, & \text{for}~ x \geq 180.
	\end{cases}
\end{eqnarray}
The proposed model densities and velocities at time, $T = 0s, 1s, 20s, 40s$ and $60s$ with $\delta$ initially fixed at $90\%$ are displayed in Figures \ref{CongestedVelocityDelta90One} and \ref{CongestedVelocityDelta90Two}. Shock wave with decreasing amplitude disappears after $20 s.$ As time progresses, densities and velocities of both vehicle classes become almost uniform moderated. On the other hand, shock wave with higher amplitude appears at all selected times for the integer order model. Densities and velocities of both vehicle classes are observed to be non uniform throughout the observation period. Higher velocities close to maximum values and low corresponding densities are observed for both vehicle classes as compared to the a case when $\delta = 20\%$ Motorcycles density is maintained at low levels. Cars density is maintained
at low levels.

\begin{figure}[H]
	
	\vspace{0pt}
	
	\centering
	\begin{subfigure}[t]{0.32\textwidth}
		\vspace{0pt}
		\includegraphics[width=\linewidth]{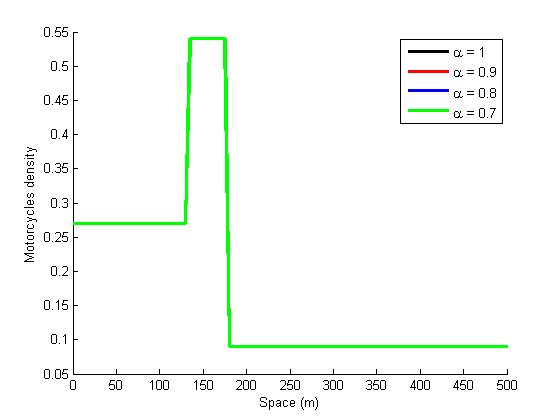}
		\caption{$\rho_m$ vs x at $T=0$}\label{CongestedDensitymDelta90_0s}
	\end{subfigure}
	\begin{subfigure}[t]{0.32\textwidth}
		\vspace{0pt}
		\includegraphics[width=\linewidth]{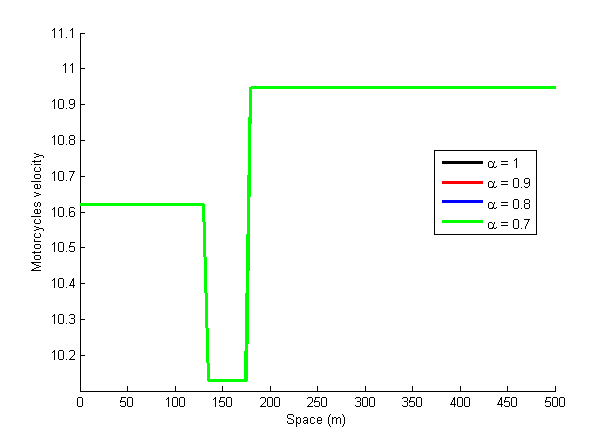}
		\caption{$v_m$ vs x at $T=0$}\label{CongestedVelocitymDelta90_0s}
	\end{subfigure}
	\begin{subfigure}[t]{0.32\textwidth}
		\vspace{0pt}
		\includegraphics[width=\linewidth]{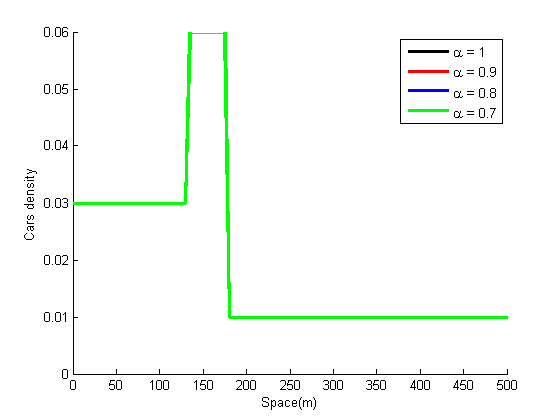}
		\caption{$\rho_c$ vs x at $T=0$}\label{CongestedDensitycDelta90_0s}
	\end{subfigure}
	\begin{subfigure}[t]{0.32\textwidth}
		\vspace{0pt}
		\includegraphics[width=\linewidth]{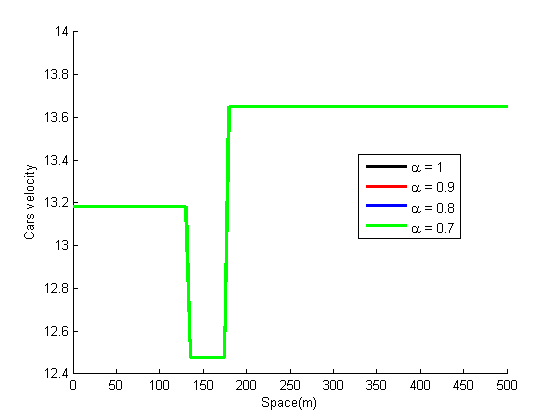}
		\caption{$v_c$ vs x at $T=0$}\label{CongestedVelocitycDelta90_0s}
	\end{subfigure}
	\begin{subfigure}[t]{0.32\textwidth}
		\vspace{0pt}
		\includegraphics[width=\linewidth]{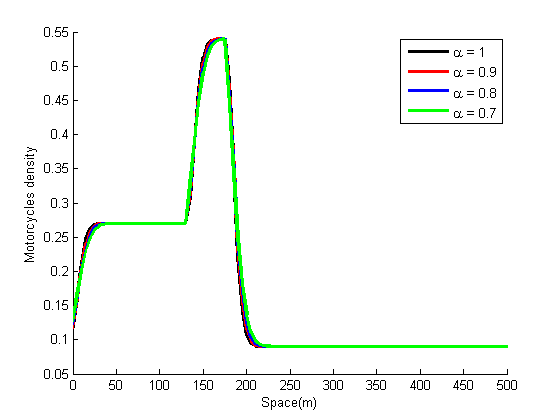}
		\caption{$\rho_m$ vs x at $T=1$}\label{CongestedDensitymDelta90_1s}
	\end{subfigure}
	\begin{subfigure}[t]{0.32\textwidth}
		\vspace{0pt}
		\includegraphics[width=\linewidth]{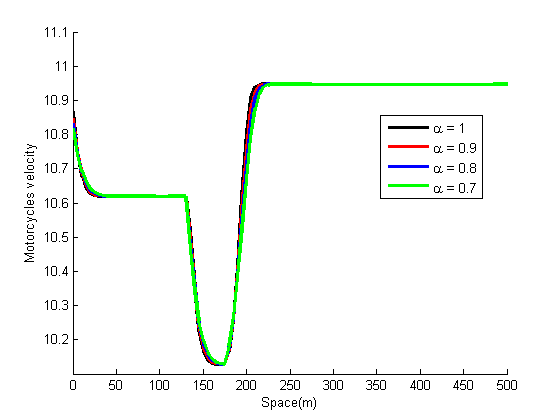}
		\caption{$v_m$ vs x at $T=1$}\label{CongestedVelocitymDelta90_1s}
	\end{subfigure}
	\begin{subfigure}[t]{0.32\textwidth}
		\vspace{0pt}
		\includegraphics[width=\linewidth]{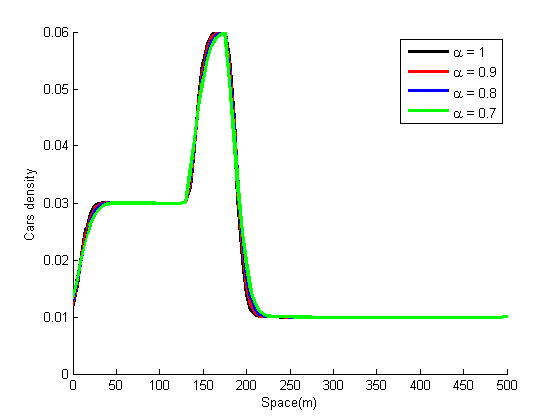}
		\caption{$\rho_c$ vs x at $T=1$}\label{CongestedDensitycDelta90_1s}
	\end{subfigure}
	\begin{subfigure}[t]{0.32\textwidth}
		\vspace{0pt}
		\includegraphics[width=\linewidth]{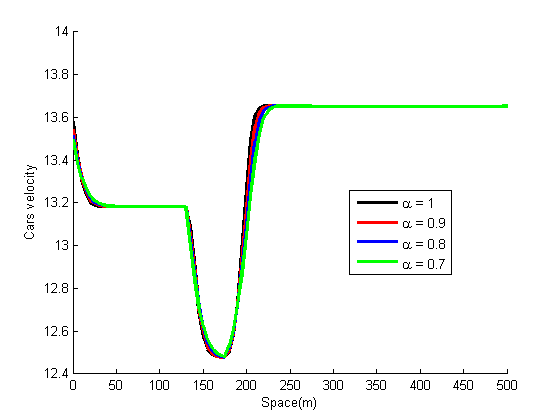}
		\caption{$v_c$ vs x at $T=1$}\label{CongestedVelocitycDelta90_1s}
	\end{subfigure}	
	\begin{subfigure}[t]{0.32\textwidth}
		\vspace{0pt}
		\includegraphics[width=\linewidth]{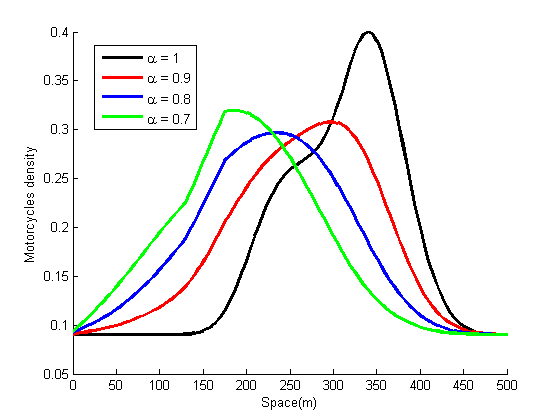}
		\caption{$\rho_m$ vs x at $T=20$}\label{CongestedDensitymDelta90_20s}
	\end{subfigure}
	\begin{subfigure}[t]{0.32\textwidth}
		\vspace{0pt}
		\includegraphics[width=\linewidth]{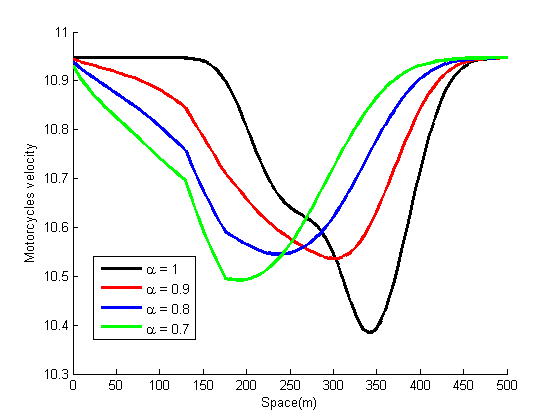}
		\caption{$v_m$ vs x at $T=20$}\label{CongestedVelocitymDelta90_20s}
	\end{subfigure}
	\begin{subfigure}[t]{0.32\textwidth}
		\vspace{0pt}
		\includegraphics[width=\linewidth]{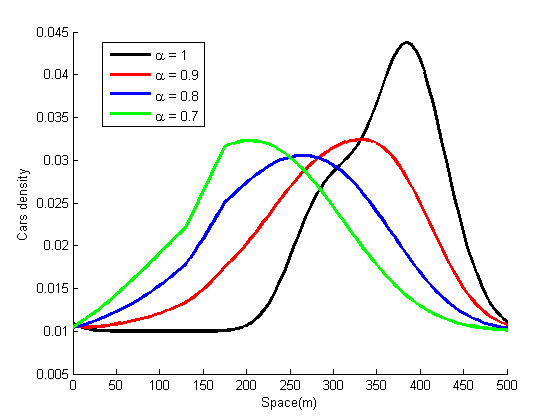}
		\caption{$\rho_c$ vs x at $T=20$}\label{CongestedDensitycDelta90_20s}
	\end{subfigure}
	\begin{subfigure}[t]{0.32\textwidth}
		\vspace{0pt}
		\includegraphics[width=\linewidth]{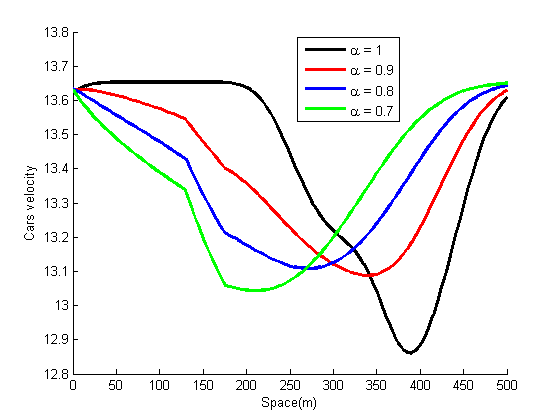}
		\caption{$v_c$ vs x at $T=20$}\label{CongestedVelocitycDelta90_20s}
	\end{subfigure}
	\caption{Densities and velocities Vs space with $\delta = 90\%,~ \alpha = 1, 0.9, 0.8, 0.7$ and $T = 0s, 1s, 20s,$ during congested traffic flow.} \label{CongestedVelocityDelta90One} 
	
	\vspace{-0.1cm}
	
\end{figure}

\begin{figure}[H]
	
	\vspace{0pt}
	
	\centering
	\begin{subfigure}[t]{0.32\textwidth}
		\vspace{0pt}
		\includegraphics[width=\linewidth]{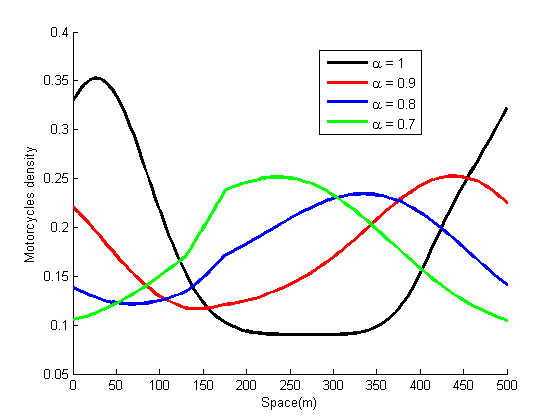}
		\caption{$\rho_m$ vs x at $T=40$}\label{CongestedDensitymDelta90_40s}
	\end{subfigure}
	\begin{subfigure}[t]{0.32\textwidth}
		\vspace{0pt}
		\includegraphics[width=\linewidth]{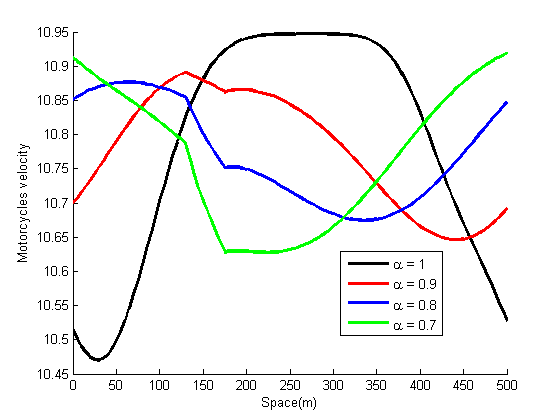}
		\caption{$v_m$ vs x at $T=40$}\label{CongestedVelocitymDelta90_40s}
	\end{subfigure}
	\begin{subfigure}[t]{0.32\textwidth}
		\vspace{0pt}
		\includegraphics[width=\linewidth]{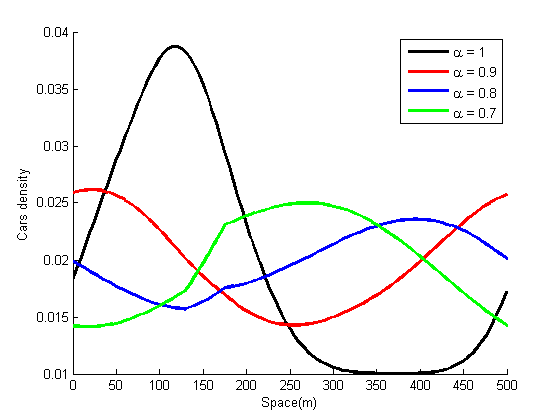}
		\caption{$\rho_c$ vs x at $T=40$}\label{CongestedDensitycDelta90_40s}
	\end{subfigure}
	\begin{subfigure}[t]{0.32\textwidth}
		\vspace{0pt}
		\includegraphics[width=\linewidth]{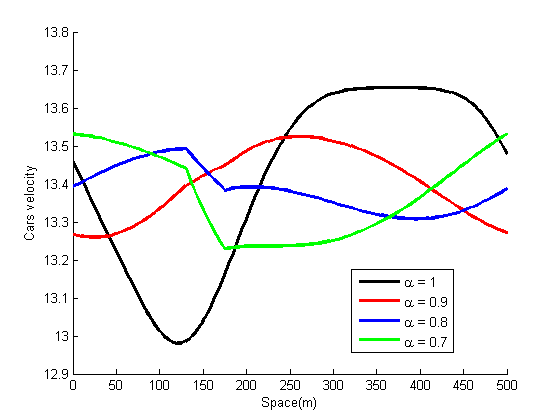}
		\caption{$v_c$ vs x at $T=40$}\label{CongestedVelocitycDelta90_40s}
	\end{subfigure}
	\begin{subfigure}[t]{0.32\textwidth}
		\vspace{0pt}
		\includegraphics[width=\linewidth]{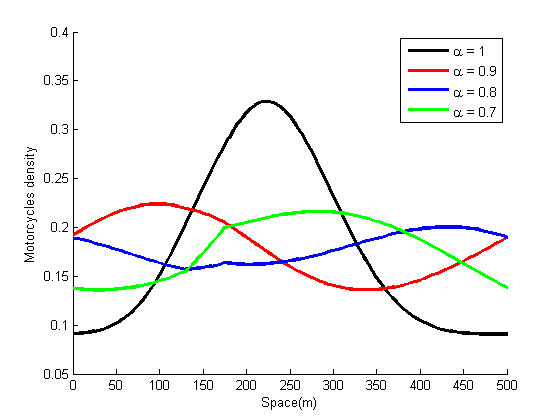}
		\caption{$\rho_m$ vs x at $T=60$}\label{CongestedDensitymDelta90_60s}
	\end{subfigure}
	\begin{subfigure}[t]{0.32\textwidth}
		\vspace{0pt}
		\includegraphics[width=\linewidth]{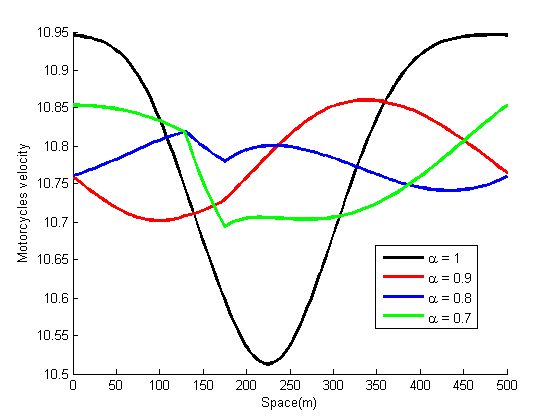}
		\caption{$v_m$ vs x at $T=60$}\label{CongestedVelocitymDelta90_60s}
	\end{subfigure}
	\begin{subfigure}[t]{0.32\textwidth}
		\vspace{0pt}
		\includegraphics[width=\linewidth]{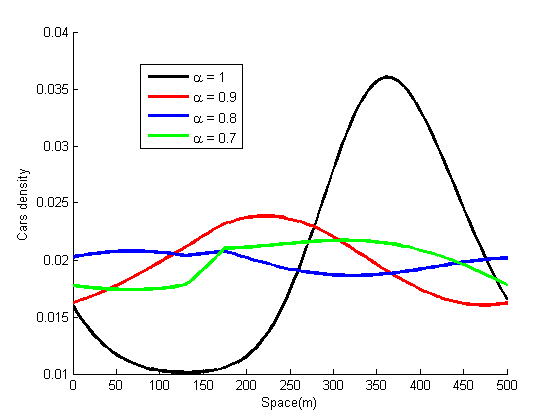}
		\caption{$\rho_c$ vs x at $T=60$}\label{CongestedDensitycDelta90_60s}
	\end{subfigure}
	\begin{subfigure}[t]{0.32\textwidth}
		\vspace{0pt}
		\includegraphics[width=\linewidth]{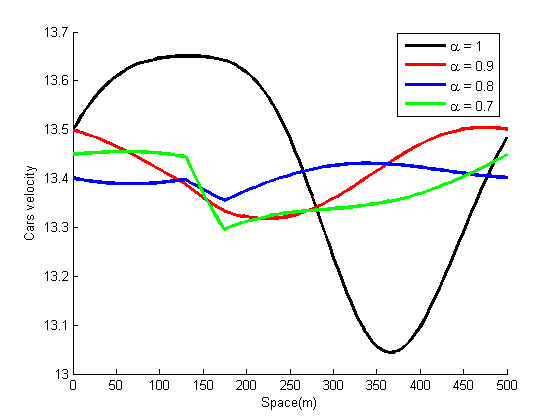}
		\caption{$v_c$ vs x at $T=60$}\label{CongestedVelocitycDelta90_60s}
	\end{subfigure}
	\caption{Densities and velocities Vs space with $\delta = 90\%,~ \alpha = 0.9, 0.8, 0.7$ and $T = 40s, 60s,$ during congested traffic flow.} \label{CongestedVelocityDelta90Two} 
	
	\vspace{-0.1cm}  
	
\end{figure}
Next $\delta$ is set to be $20\%,$ that is; many cars and fewer motorcycles. The proposed model density and velocity distributions are displayed by Figures \ref{CongestedDelta20One} and \ref{CongestedDelta20Two}. 
Same results are observed as in the previous case, except that here both vehicle classes move with lower velocities. 

\begin{figure}[H]
	
	\vspace{0pt}
	
	\centering
	\begin{subfigure}[t]{0.32\textwidth}
		\vspace{0pt}
		\includegraphics[width=\linewidth]{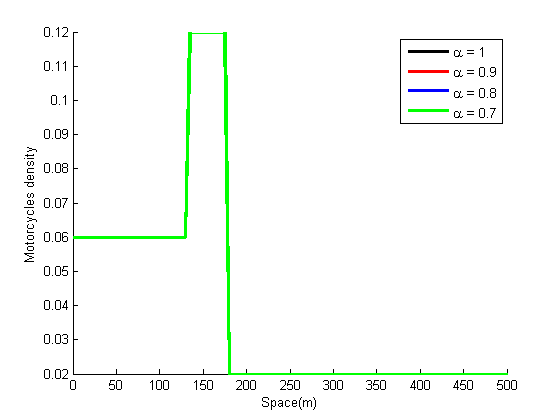}
		\caption{$\rho_m$ vs x at $T=0s$}\label{CongestedDensitymDelta20_0s}
	\end{subfigure}
	\begin{subfigure}[t]{0.32\textwidth}
		\vspace{0pt}
		\includegraphics[width=\linewidth]{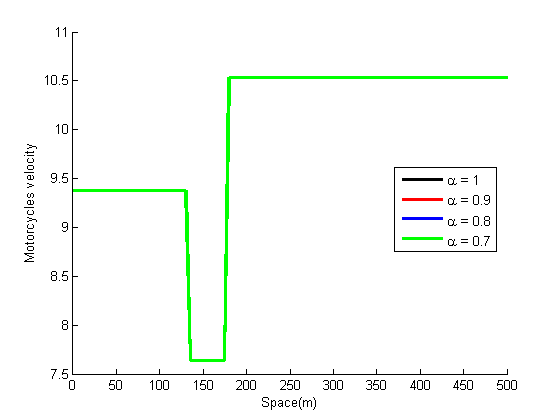}
		\caption{$v_m$ vs x at $T=0s$}\label{CongestedVelocitymDelta20_0s}
	\end{subfigure}	
	\begin{subfigure}[t]{0.32\textwidth}
		\vspace{0pt}
		\includegraphics[width=\linewidth]{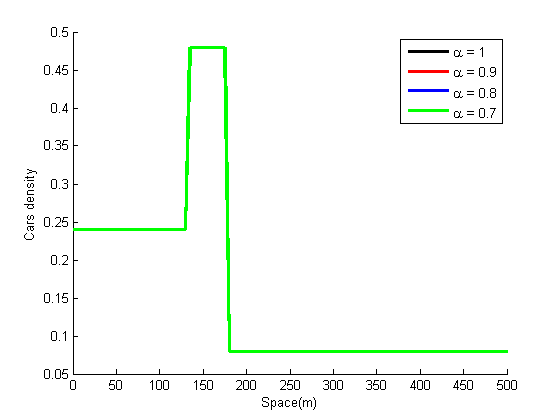}
		\caption{$\rho_c$ vs x at $T=0s$}\label{CongestedDensitycDelta20_0s}
	\end{subfigure}
	\begin{subfigure}[t]{0.32\textwidth}
		\vspace{0pt}
		\includegraphics[width=\linewidth]{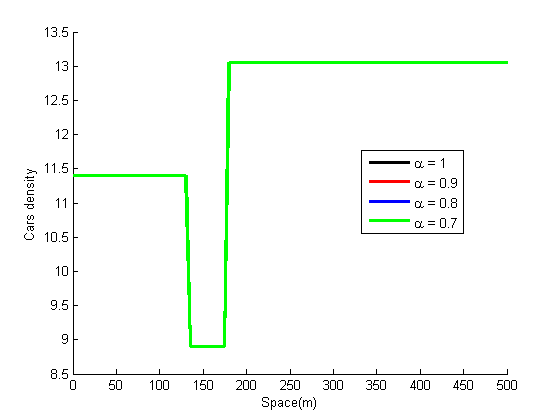}
		\caption{$v_c$ vs x at $T=0s$}\label{CongestedVelocitycDelta20_0s}
	\end{subfigure}
	\begin{subfigure}[t]{0.32\textwidth}
		\vspace{0pt}
		\includegraphics[width=\linewidth]{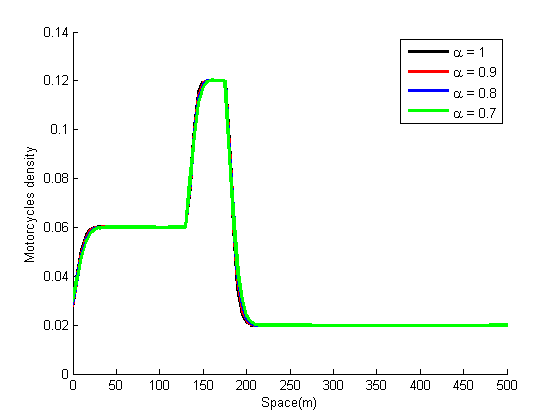}
		\caption{$\rho_m$ vs x at $T=1s$}\label{CongestedDensitymDelta20_1s}
	\end{subfigure}
	\begin{subfigure}[t]{0.32\textwidth}
		\vspace{0pt}
		\includegraphics[width=\linewidth]{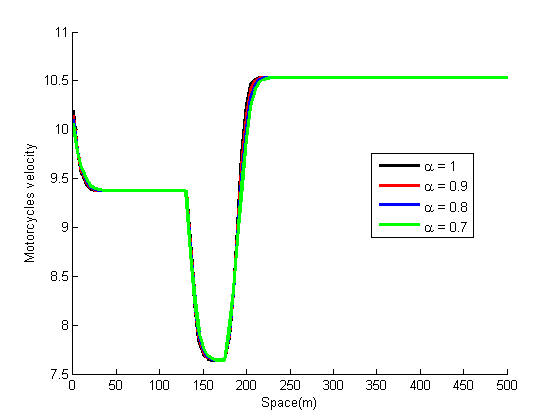}
		\caption{$v_m$ vs x at $T=1s$}\label{CongestedVelocitymDelta20_1s}
	\end{subfigure}
	\begin{subfigure}[t]{0.32\textwidth}
		\vspace{0pt}
		\includegraphics[width=\linewidth]{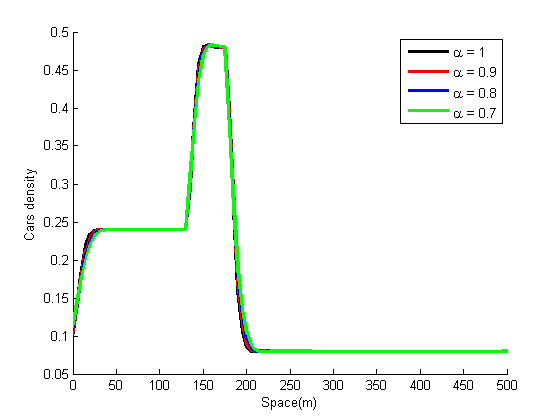}
		\caption{$\rho_c$ vs x at $T=1s$}\label{CongestedDensitycDelta20_1s}
	\end{subfigure}
	\begin{subfigure}[t]{0.32\textwidth}
		\vspace{0pt}
		\includegraphics[width=\linewidth]{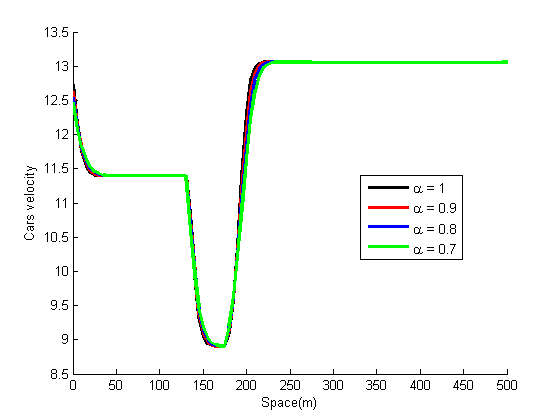}
		\caption{$v_c$ vs x at $T=1s$}\label{CongestedVelocitycDelta20_1s}
	\end{subfigure}
	\begin{subfigure}[t]{0.32\textwidth}
		\vspace{0pt}
		\includegraphics[width=\linewidth]{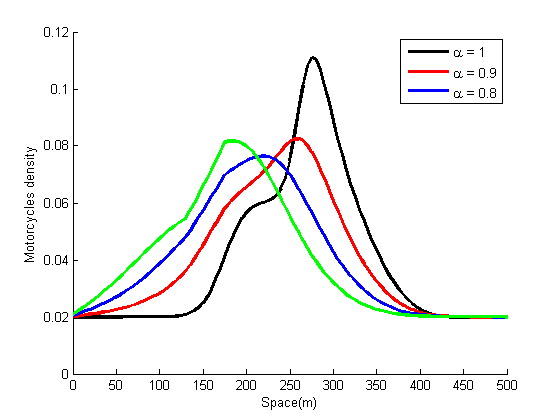}
		\caption{$\rho_m$ vs x at $T=20s$}\label{CongestedDensitymDelta20_20s}
	\end{subfigure}
	\begin{subfigure}[t]{0.32\textwidth}
		\vspace{0pt}
		\includegraphics[width=\linewidth]{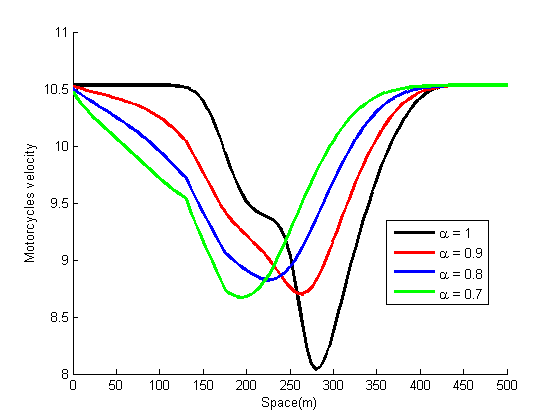}
		\caption{$v_m$ vs x at $T=20s$}\label{CongestedVelocitymDelta20_20s}
	\end{subfigure}
	\begin{subfigure}[t]{0.32\textwidth}
		\vspace{0pt}
		\includegraphics[width=\linewidth]{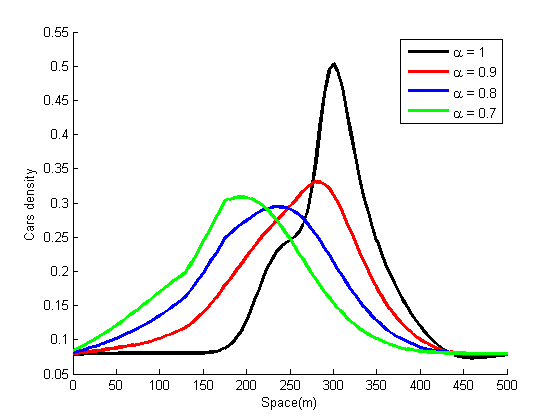}
		\caption{$\rho_c$ vs x at $T=20s$}\label{CongestedDensitycDelta20_20s}
	\end{subfigure}
	\begin{subfigure}[t]{0.32\textwidth}
		\vspace{0pt}
		\includegraphics[width=\linewidth]{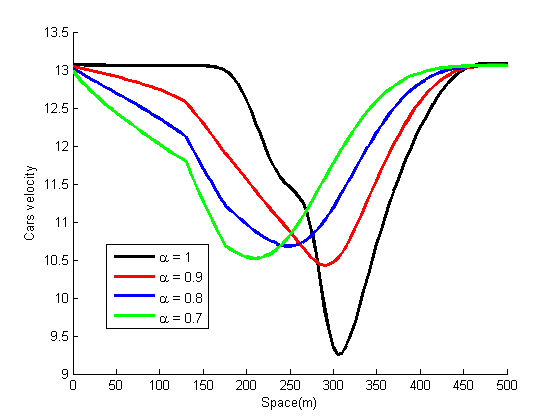}
		\caption{$v_c$ vs x at $T=20s$}\label{CongestedVelocitycDelta20_20s}
	\end{subfigure}
	\caption{Densities and velocities Vs space with $\delta = 20\%,$  and $\alpha = 1, 0.9, 0.8, 0.7,$ at $T = 0s, 1s, 20s,$ during congestion.} \label{CongestedDelta20One} 
	
	\vspace{-0.1cm}
	
\end{figure}

\begin{figure}[H]
	
	\vspace{0pt}
	
	\centering
	\begin{subfigure}[t]{0.32\textwidth}
		\vspace{0pt}
		\includegraphics[width=\linewidth]{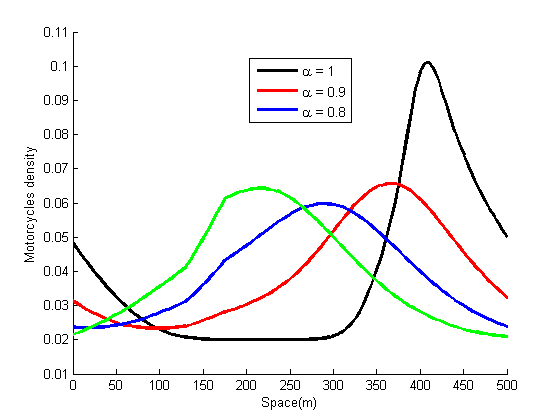}
		\caption{$\rho_m$ vs x at $T=40s$}\label{CongestedDensitymDelta20_40s}
	\end{subfigure}
	\begin{subfigure}[t]{0.32\textwidth}
		\vspace{0pt}
		\includegraphics[width=\linewidth]{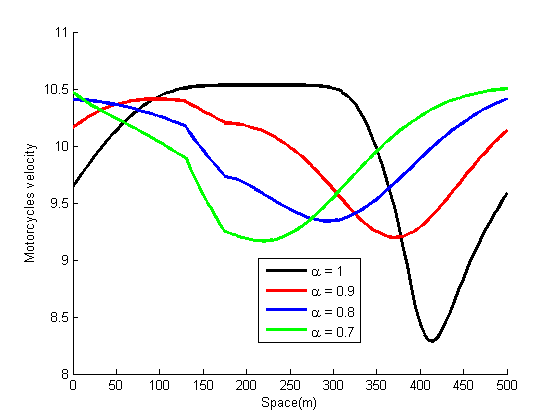}
		\caption{$v_m$ vs x at $T=40s$}\label{CongestedVelocitymDelta20_40s}
	\end{subfigure}
	\begin{subfigure}[t]{0.32\textwidth}
		\vspace{0pt}
		\includegraphics[width=\linewidth]{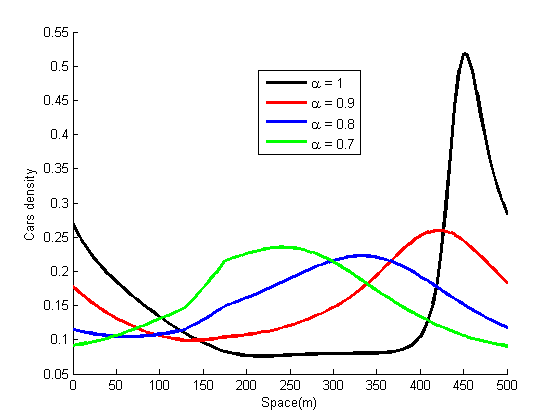}
		\caption{$\rho_c$ vs x at $T=40s$}\label{CongestedDensitycDelta20_40s}
	\end{subfigure}
	\begin{subfigure}[t]{0.32\textwidth}
		\vspace{0pt}
		\includegraphics[width=\linewidth]{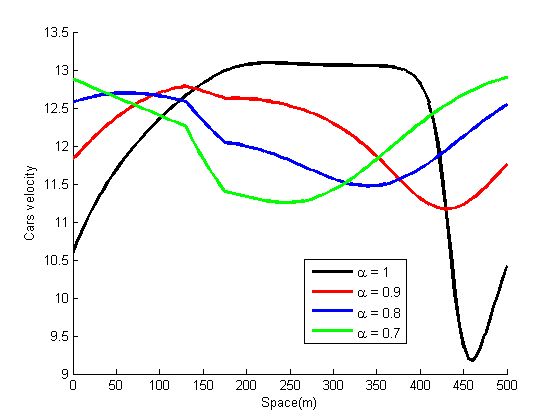}
		\caption{$v_c$ vs x at $T=40s$}\label{CongestedVelocitycDelta20_40s}
	\end{subfigure}
	\begin{subfigure}[t]{0.32\textwidth}
		\vspace{0pt}
		\includegraphics[width=\linewidth]{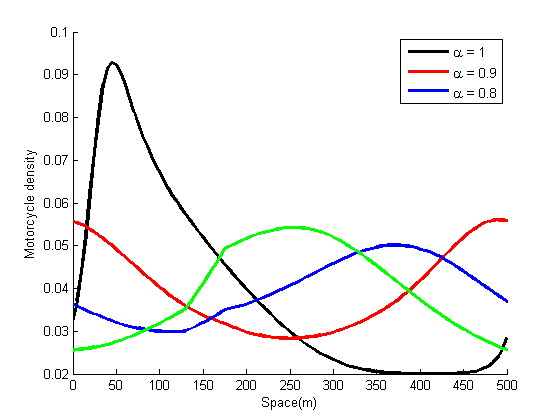}
		\caption{$\rho_m$ vs x at $T=60s$}\label{CongestedDensitymDelta20_60s}
	\end{subfigure}
	\begin{subfigure}[t]{0.32\textwidth}
		\vspace{0pt}
		\includegraphics[width=\linewidth]{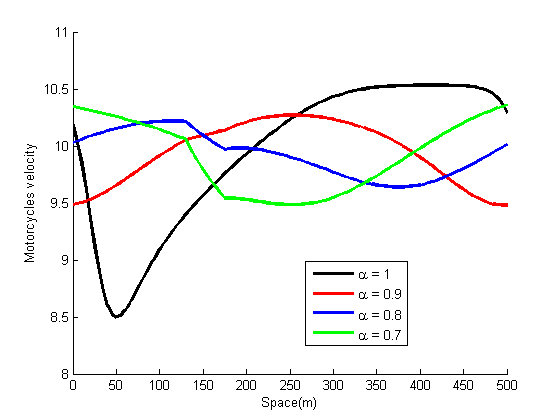}
		\caption{$v_m$ vs x at $T=60s$} \label{CongestedVelocityDelta20mDelta20_60s}
	\end{subfigure}
	\begin{subfigure}[t]{0.32\textwidth}
		\vspace{0pt}
		\includegraphics[width=\linewidth]{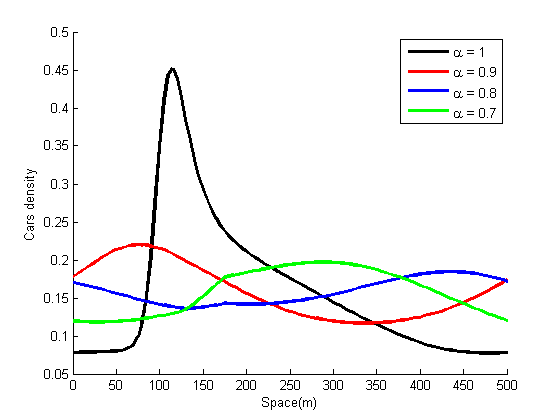}
		\caption{$\rho_c$ vs x at $T=60s$}\label{CongestedDensitycDelta20_60s}
	\end{subfigure}
	\begin{subfigure}[t]{0.32\textwidth}
		\vspace{0pt}
		\includegraphics[width=\linewidth]{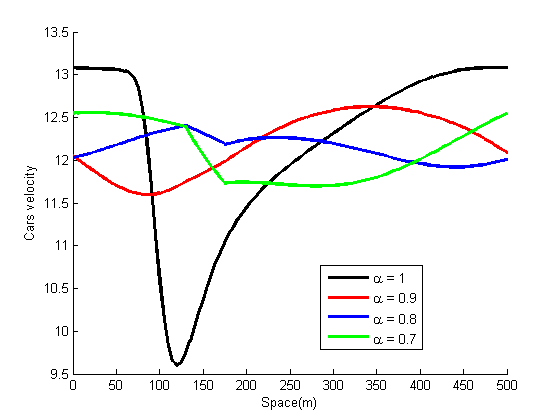}
		\caption{$v_c$ vs x at $T=60s$} \label{CongestedVelocitycDelta20_60s}
	\end{subfigure}
	
	\caption{Densities and velocities Vs space with $\delta = 20\%,$ $\alpha = 1, 0.9, 0.8, 0.7,~ \textbf{and}~ T = 20s, 40s, 60s,$ during congestion.} \label{CongestedDelta20Two} 
	
	\vspace{-0.1cm}
	
\end{figure}

\section{Discussion}
The aim of this paper is to determine the effect of fractional order derivatives of time on the flow of heterogeneous vehicular traffic. An explicit difference scheme is obtained through finite difference method of discretization. The scheme is proved to be consistent, conditionally stable and convergent. Numerical flux is computed by original Roe decomposition and an entropy condition applied to the Roe decomposition. Qualitative analysis shows that the results obtained are a natural extension and generalization of the integer order model, which can be used for reference in solving other macroscopic time-fractional traffic flow models. Numerical results show that smooth results are attainable when the derivative order is either integer or close to integer. During freeway, both vehicle classes move with high velocities close to their maximum values irrespective of the derivative order. When the road is congested, a shock wave develops at $1 s,$ for all derivative orders. Its amplitude continues to decrease with increase of time until it disappears after $20 s,$ for the case of fractional orders. At all times, a shock wave is sustained for the integer order derivative. Same results are obtained at $1 s,$ irrespective the derivative order. In both freeway and congestion, vehicles move with higher velocities when cars are fewer on the road and vice versa. The most desirable results are obtained when the fractional order is close to 1 (the integer order). Densities and velocities for each vehicle class become smoother over time
and are shown to remain within limits. Thus, the results obtained from the proposed model are realistic. As
a control for traffic jam, enforcement of lane discipline and enlargement of roadways are recommended. 

In general, while investigating the heterogeneous traffic flow problems, a fractional-order model is recommended because it maintains moderated values of density and velocity throughout the simulation period unlike the integer-order model whose values fluctuate so much. Although some theoretical results and simulations are presented in this paper, there is much more work unexplored along this topic, such as existence and uniqueness of solution for the proposed model.  

\section*{Acknowledgment}
	The first author acknowledges her other research supervisors Prof. Semu Mitiku Kassa of Department of Mathematics and Statistical Sciences, Botswana
International University of Science and Technology (BIUST), P/Bag $16,$ Palapye, Botswana for providing technical guidance and mentorship. The first author also acknowledges the Sida bilateral program with Makerere University, 2015-2020, project 316 ``Capacity building in Mathematics and its applications'' for providing financial support.

\end{document}